\documentclass[reqno]{amsart}
\usepackage{amssymb}
\usepackage[usenames, dvipsnames]{color}
\usepackage{enumitem}

\usepackage{mathrsfs}

\usepackage{marginnote}
\usepackage{todonotes}
\usepackage{hyperref}
\usepackage{esint}
\usepackage{verbatim}

\theoremstyle{plain}
\newtheorem{theorem}{Theorem}[section]
\newtheorem{lemma}[theorem]{Lemma}
\newtheorem{corollary}[theorem]{Corollary}

\theoremstyle{definition}
\newtheorem{definition}[theorem]{Definition}

\theoremstyle{remark}
\newtheorem{remark}[theorem]{Remark}

\numberwithin{equation}{section}

\newcommand{\bN}{\mathbb{N}}

\newcommand{\bR}{\mathbb{R}}
\newcommand{\bH}{\mathbb{H}}

\newcommand\cD{\mathcal{D}}

\newcommand\cU{\mathcal{U}}

\newcommand\sW{\mathscr{W}}

\newcommand\sH{\mathscr{H}}

\makeatletter
\@namedef{subjclassname@2020}{%
  \textup{2020} Mathematics Subject Classification}
\makeatother

\begin{document}

\title[Dirichlet problem, Parabolic equations, Degenerate coefficients]{On a class of divergence form linear parabolic equations with degenerate coefficients}

\author[T. Phan]{Tuoc Phan}
\address[T. Phan]{Department of Mathematics, University of Tennessee, 227 Ayres Hall,
1403 Circle Drive, Knoxville, TN 37996-1320, USA}
\email{phan@math.utk.edu}

\author[H. V. Tran]{Hung Vinh Tran}
\address[H. V. Tran]{Department of Mathematics, University of Wisconsin-Madision, Van Vleck Hall
480 Lincoln Drive
Madison, WI  53706, USA}
\email{hung@math.wisc.edu}

\thanks{T. Phan is partially supported by the Simons Foundation, grant \# 354889.
H. Tran is supported in part by NSF CAREER grant DMS-1843320 and a Simons Fellowship.
}

\begin{abstract} 
We study a class of linear parabolic equations in divergence form with degenerate coefficients on the upper half space.
Specifically, the equations are considered in $(-\infty, T) \times \bR^d_+$, where $\bR^d_+ = \{x \in \bR^d\,:\, x_d>0\}$ and $T\in {(-\infty, \infty]}$ is given, and the diffusion matrices are the product of $x_d$ and bounded uniformly elliptic matrices, which are degenerate at $\{x_d=0\}$. 
As such, our class of equations resembles well the corresponding class of degenerate viscous Hamilton-Jacobi equations.
We obtain wellposedness results and regularity type estimates in some appropriate weighted Sobolev spaces for the solutions.
 \end{abstract}

\subjclass[2020]{35K65, 35K67, 35K20, 35D30}
\keywords{Degenerate  parabolic equations in divergence form; boundary regularity estimates; existence and uniqueness; weighted Sobolev spaces}

\maketitle
\section{Introduction}
\subsection{Settings}
Let $T \in (-\infty, \infty]$, $d \in \bN$ and denote by
\[
\bR^d_+ = \{x \in \bR^d\,:\, x_d>0\} \quad \text{ and } \quad \Omega_T = (-\infty, T) \times \bR^d_+.
\]
Let  $(a_{ij})_{i,j=1}^d: \Omega_T \rightarrow \mathbb{R}^{d\times d}$ be a matrix of measurable functions satisfying the following ellipticity  and boundedness conditions 
\begin{equation} \label{elli}
\nu|\xi|^2 \leq a_{ij}(t,x) \xi_i \xi_j, \quad |a_{ij}(t,x)| \leq \nu^{-1},  \quad (t,x) \in \Omega_T 
\end{equation}
for $\xi =(\xi_1, \xi_2, \ldots, \xi_d) \in \bR^d$,  where $\nu \in (0,1)$ is given.
Throughout the paper, we write $z= (t,x)  \in \Omega_T$ with $x=(x', x_d) \in \bR^{d-1} \times \bR_+$ where $x' = (x_1, x_2, \ldots, x_{d-1}) \in \bR^{d-1}$.  
We study the following class of  parabolic equations in divergence form with degenerate coefficients
\begin{equation} \label{main-eqn}
 u_t(z) + \lambda c_0(z) u(z) - x_d D_i\big(a_{ij}(z) D_{j} u(z) - F_i\big) = \sqrt{\lambda}  f(z) \quad \text{in} \quad \Omega_T
\end{equation}
under the homogeneous Dirichlet boundary condition
\begin{equation} \label{main-bdr-cond}
u = 0 \quad \text{on} \quad (-\infty, T) \times \partial \bR^d_+.
\end{equation}
Here, $\lambda \geq 0$ and $F_i : \Omega_T \rightarrow \mathbb{R}$, $f: \Omega_T \rightarrow \mathbb{R}$ are given measurable functions. Moreover,  $c_0 : \Omega_T \rightarrow \bR$ is a given measurable function satisfying the boundedness condition
\begin{equation} \label{a-c.conds}
\nu \leq c_0(z) \leq \nu^{-1}, \quad  z = (t,x) \in \Omega_T.
\end{equation}
We note that the equation \eqref{main-eqn} can be written in the form
\begin{equation} \label{main-PDE-sing}
x_d^{-1}\big(u_t(z) + \lambda c_0(z) u(z)\big) - D_i\big(a_{ij}(z) D_{j} u(z) - F_i\big) = \sqrt{\lambda}  x_d^{-1}f(z) \quad \text{in} \quad \Omega_T
\end{equation}
in which the coefficients become singular on the boundary $\{x_d =0\}$ of the considered domain. 
The weight $x_d^{-1}$ appearing in the coefficients of \eqref{main-PDE-sing} is not in the class of $A_2$ Muckenhoupt weights as commonly assumed in literature, and it is even not locally integrable near $\{x_d =0\}$.

The main goal in this paper is to find suitable Sobolev spaces in which the wellposedness and regularity estimates for solutions of \eqref{main-eqn}-\eqref{main-bdr-cond} are established. 
In our main result (Theorem \ref{main-thrm} below), under an assumption on the smallness of the partial mean oscillations of the coefficients in small balls, the following estimate is proved
\begin{equation}\label{show-off}
 \|Du\|_{L_p(\Omega_T)} + \sqrt{\lambda} \|u\|_{L_p(\Omega_T, x_d^{-p/2})}   \leq N \Big[ \|F\|_{L_p(\Omega_T)} + \|f\|_{L_2(\Omega_T, x_d^{-p/2})} \Big],
\end{equation}
for every weak solution $u$ of \eqref{main-eqn}-\eqref{main-bdr-cond} and for $\lambda>0$ sufficiently large, where $p \in (1, \infty)$, $N = N(d, \nu, p)>0$ is some generic constant, and $L_p(\Omega_T, \omega)$ denotes weighted Lebesgue space with weight $\omega$. 
To the best of our knowledge, this paper is the first one in which wellposedness and regularity estimates in Sobolev spaces for the class of equation \eqref{main-eqn} is studied. 
The estimate \eqref{show-off} is therefore completely new.

\subsection{Motivations and relevant literature}
Our main object, equation \eqref{main-eqn}, in its simplest form reads
\begin{equation} \label{simplest-eqn}
 u_t(z) + \lambda u(z) - x_d \Delta u = \sqrt{\lambda}  f(z) \quad \text{in} \quad \Omega_T,
\end{equation}
which resembles well the following degenerate viscous Hamilton-Jacobi equation
\begin{equation} \label{HJ-eqn}
 u_t(z) + \lambda u(z) +H(z,Du)- x_d \Delta u = 0  \quad \text{in} \quad \Omega_T,
\end{equation}
where $H:\Omega_T \times \bR^d \to \bR$ is a given Hamiltonian.
Note that the viscosity coefficient of \eqref{HJ-eqn}  not only is degenerate at $\{x_d=0\}$ but also has linear growth as $x_d \to \infty$.
For typical viscous Hamilton-Jacobi equations with possibly degenerate and bounded diffusions, one often has uniqueness of viscosity solutions, and such solutions are often Lipschitz in $z$ (see \cite{CIL, AT} and the references therein).
Finer regularity of solutions is not very well understood in the literature.
In particular, optimal regularity of solutions to \eqref{HJ-eqn} near $\{x_d=0\}$ has not been investigated.

It is of our goals to study wellposedness and regularity of solutions to  \eqref{main-eqn} (and \eqref{simplest-eqn}), which will pave the way for us to study wellposedness and regularity of solutions to  \eqref{HJ-eqn} later.
We note that the linear growth of the diffusion coefficients at infinity has to be handled carefully.
More specifically, we obtain wellposedness results, $W^{1,p}$ regularity type estimates  for solutions to  \eqref{main-eqn}, and $W^{2,p}$ regularity type estimates  for solutions to \eqref{simplest-eqn} in some appropriate weighted norms.

\medskip

We would like to mention that the literature on regularity theory for degenerate elliptic and parabolic equations is vast.
A particularly relevant equation that was studied much in the literature is
\begin{equation} \label{prototype-eqn}
 u_t(z) + \lambda u(z) - x_d \Delta u -\beta D_d u=f(z) \quad \text{in} \quad \Omega_T.
\end{equation}
Here, $\lambda \geq 0$ and $\beta>0$ are given constants.
Although \eqref{simplest-eqn} and \eqref{prototype-eqn} are quite similar, the requirement that $\beta>0$ is essential in the analysis of  \eqref{prototype-eqn}.
Equation  \eqref{prototype-eqn} is an important model equation appearing in the study of porous media equations and parabolic Heston equations for example.
We refer the readers to \cite{DaHa, Koch, FePo} for the wellposedness and regularity results of \eqref{prototype-eqn} and more general equations of this type.
The Schauder a priori estimates for solutions in weighted H\"older spaces were obtained in \cite{DaHa, FePo}; 
and weighted $W^{2,p}$ estimates for solutions were obtained in \cite{Koch}.
A remarkable point is that the boundary condition of \eqref{prototype-eqn} on $\{x_d=0\}$ may be omitted thanks to its special features.
In contrast, we impose the homogenous Dirichlet boundary condition $u=0$ on $\{x_d=0\}$ for \eqref{main-eqn} and \eqref{simplest-eqn}. 
See Theorem \ref{zero-trace-1} below about the trace of our functional spaces for more information related to the boundary condition.
Naturally, our methods and obtained $W^{1,p}$, $W^{2,p}$ estimates are rather different from those in  \cite{DaHa, Koch, FePo} with different weights, and according to the best of our knowledge, they are new in the literature.

\smallskip
{We also point out that similar results on wellposedness and regularity estimates in weighted Sobolev spaces for equations with singular-degenerate coefficients were established recently in the series of papers \cite{Dong-Phan, Dong-Phan-1, Dong-Phan-2, Dong-Phan-3}. 
Note that in the classes of equations studied in these papers, the weights of singularity/degeneracy arise in a balanced way in both coefficients  of $u_t$ and of $Du$, and this balance is important for the analysis and functional space settings  in \cite{Dong-Phan, Dong-Phan-1,  Dong-Phan-2, Dong-Phan-3}. 
This important structure was pointed out in the classical work \cite{Chi-Se-1, Chi-Se-2} in which  Harnack's inequalities were proved to be false in certain cases if the balance of the weights in the coefficients is lost. 
Because of this fact and due to the structure of \eqref{main-eqn}, the Lipschitz regularity estimates for solutions of homogeneous equations \eqref{main-eqn} are very delicate, and maybe not expected in general perspectives. 
The novelty in our work is to use  $(x_d+\epsilon)^{-1}u$ as a test function for \eqref{main-PDE-sing} with $\epsilon>0$ to discover a hidden special form of energy estimate for \eqref{main-eqn} (see Lemma \ref{w-lemma-0529} below). 
From this, we apply an iteration technique using anisotropic Sobolev embeddings and Hardy's inequality to derive Lipschitz regularity estimates for solutions. 
This is done in Section \ref{subsec:hom}. 
It is important to note that the method still works for systems of equations.}

\smallskip

Besides, we would like mention the classical papers \cite{Fabes, FKS} in which H\"{o}lder's regularity estimates are proved for classes of elliptic equations in which the coefficients are singular and degenerate as $A_2$-Muckenhoupt weights. See also \cite{ MuSt1, MuSt2} for earlier results with more restrictions on the weights.
The $W^{1,p}$-counterpart of the H\"{o}lder's regularity in \cite{Fabes, FKS} was recently established in \cite{CMP} with some additional smallness condition on the weighted mean oscillation of coefficients, which is only valid for sufficiently small $\alpha$ when considering the particular case with the weight $|x|^\alpha$. 
Note also that the classes of equations studied in these mentioned papers are different from \eqref{main-eqn}. 
Moreover, our weight in \eqref{main-eqn} or \eqref{main-PDE-sing} is not an $A_2$-weight, which is an essential assumption used in \cite{Fabes, FKS, CMP}.  
For other classical results that are closely related to our study, we refer to \cite{Fichera, KoNi, OR}.
See also \cite{Lin, Wang2} for other interesting work on regularity estimates of equations with singular degenerate coefficients appearing in geometric analysis.

\medskip

The rest of the paper is organized as follows.
In Section \ref{sec:FA-Results}, we introduce needed functional spaces and state our main results.
Section \ref{sec:L2} is devoted to the study of $L_2$-weak solutions.
Then, in Section \ref{sec:xd}, we study equations with coefficients depending only on $x_d$.
A crucial part of the analysis lies in the pointwise estimates for homogeneous equations in Section \ref{subsec:hom}.
Finally, the proofs of our main results (Theorem \ref{main-thrm} and Corollary \ref{cor-2}) are given in Section \ref{sec:proof-main}.

\section{Functional spaces and statements of main results}\label{sec:FA-Results}
\subsection{Functional spaces and definition of weak solutions}  
For $p \in [1, \infty)$,  $-\infty\le S<T\le +\infty$, and $\cD \subset \bR^d_+$, let $L_p((S,T)\times \cD)$ be  the usual Lebesgue space consisting of measurable functions $u$ on $(S,T)\times \cD$ such that the norm
\[
\|u\|_{L_p( (S,T)\times \cD)}= \left( \int_{(S,T)\times \cD} |u(t,x)|^p\, dxdt \right)^{1/p} <\infty.
\]
For $p\in [1,\infty)$,  and for given weight $\omega$ defined on $(S,T)\times \cD$, we define $L_{p}((S,T)\times \cD,\omega)$ to be the weighted Lebesgue space on $(S,T)\times \cD$ equipped with the norm
\begin{equation*}
\|u\|_{L_{p}((S,T)\times \cD, \omega)}=\left(\int_{S}^T\int_{\cD} |u(t,x)|^p \omega (t,x)\, dx dt\right)^{1/p}.
\end{equation*}
We define the weighted Sobolev space
$$
W^1_p(\cD)=\big\{u\in L_p(\cD, x_d^{-p/2}):\,Du\in L_p(\cD)\big\}
$$
that is equipped with the norm
$$
\|u\|_{W^1_p(\cD)}=\|u\|_{L_p(\cD, x_d^{-p/2})}+\|Du\|_{L_p(\cD)}.
$$
We note that $W^1_p(\cD)$ is different from the usual Sobolev space due to the availability of the weight $x_d^{-p/2}$ in the $L_p$-norm of $u$. As in the standard way, we denote $\sW^1_p(\cD)$ the Sobolev space defined to be the closure in $W^1_p(\cD)$ of all compactly supported functions in $C^\infty(\overline{\cD})$ vanishing near $\overline{\cD} \cap \{x_d=0\}$ if $\overline{\cD} \cap \{x_d=0\}$ is not empty. 
The space $\sW^1_p(\cD)$ is equipped with the norm
$$
\|u\|_{\sW^1_p(\cD)}=\|u\|_{L_p(\cD, x_d^{-p/2})}+\|Du\|_{L_p(\cD)}.
$$
Regarding the relation between the functional spaces $W^1_p(\cD)$ and $\sW^1_p(\cD)$, we have the following trace theorem. 
Though it is not used in the paper, it is important to point out. 
The proof of the theorem is given in Appendix \ref{sec-proof-trace}.

\begin{theorem} \label{zero-trace-1}  
Let $B_1'$ be the unit ball in $\bR^{d-1}$ centered at the origin, $\cD= B_1' \times (0, 1)$, and $p \in [2, \infty)$. 
If $u \in W^1_p(\cD)$, then $u(x', 0) =0$ in the trace sense  for a.e. $x' \in B_1'$.
 Moreover, 
 \[
 W^1_p(\cD) = \sW^{1}_p (\cD).
 \]
\end{theorem} 
Next, due to the structure of \eqref{main-eqn}, we define the dual space
\[
\begin{split}
& \bH_{p}^{-1}( (S,T)\times \cD) \\
& =\big\{u\,:\, u  =  x_d D_iF_i +f \ \ \text{for some}\ f\in L_{p}( (S,T)\times \cD, x_d^{-p/2}) \text{ and }\\
& \qquad \quad \ F= (F_1,\ldots,F_d) \in L_{p}((S,T)\times \cD)^{d}\big\},
\end{split}
\]
that is equipped with the norm
\begin{align*}
&\|u\|_{\bH_{p}^{-1}((S,T)\times \cD, \omega)} \\
&=\inf\big\{\|F\|_{L_{p}((S,T)\times \cD)}
+\|f\|_{L_{p}((S,T)\times \cD, x_d^{-p/2})}\,:\, u= x_d D_iF_i +f\big\}.
\end{align*}
Then, we define the space
\[
 \sH_{p}^1((S,T)\times \cD)
 =\big\{u \in L_p((S, T),  \sW^1_p(\cD)):    u_t\in  \bH_{p}^{-1}( (S,T)\times \cD)\big\}
\]
which is equipped with the norm
\begin{align*}
\|u\|_{\sH_{p}^1((S,T)\times \cD)} &= \|u\|_{L_{p}((S,T)\times \cD, x_d^{-p/2})} + \|Du\|_{L_{p}((S,T)\times \cD)} \\
& \qquad +\|u_t\|_{\bH_{p}^{-1}((S,T)\times \cD)}.
\end{align*}
Now, we give the definition of weak solutions for the class of equations \eqref{main-eqn}.
\begin{definition}  \label{weak-form} 
Let $p \in (1, \infty)$, $F \in L_p((S,T)\times \cD)^d$ and $f \in L_p((S,T)\times \cD, x_d^{-p/2})$. 
We say that $u \in \sH_{p}^1((S,T)\times \cD)$ is a weak solution to \eqref{main-eqn} in $(S,T)\times \cD$ with boundary condition $u =0$ on $\overline{\cD} \cap \{x_d =0\}$ when $\overline{\cD} \cap \{x_d =0\} \not=\emptyset$ if
\begin{align*} 
&\int_{(S,T)\times \cD} x_d^{-1}(-u \partial_t \varphi +\lambda c_0(z)u \varphi)\, dz + \int_{(S,T)\times \cD}( a_{ij} D_{j} u  - F_i)D_{i} \varphi\, dz \\  
&= \lambda^{1/2} \int_{(S,T)\times \cD} x_d^{-1} f(z) \varphi(z)\, dz,
\end{align*}
for any $\varphi \in C_0^\infty((S,T)\times \cD)$.
\end{definition}

\subsection{Notations and statement of main results} 
We need some notations to state our results. 
Throughout the paper, for $r >0$ and  $z_0 = (t_0, x_0) \in \bR \times \overline{\bR}^d_+$, we denote $B_r(x_0)$ the ball in $\bR^d$ of radius $r$ centered at $x_0$. 
Moreover, the upper half balls in $\bR^d_+$ centered at $x_0$ and in $\bR^{d+1}_{+}$ centered at $z_0$ are respectively defined by
\[
B_r^+(x_0) = B_r(x_0) \cap \{ x_d >0\} \quad \text{ and } \quad Q_r^+(z_0) = (t_0 - r, t_0] \times B_r^+(x_0).
\] 
For $x_0' \in \bR^{d-1}$, we denote $B'_r(x_0')$ the ball in $\bR^{d-1}$ of radius $r$ and centered at $x_0'$. 
Similarly, for $z_0' = (t_0, x_0') \in \bR \times \bR^{d-1}$, we also denote
\[
Q_r'(z_0') = (t_0 - r, t_0] \times B_r'(x_0').
\]
When $z_0 =0$ or $x_0=0$, we simply write $B_r = B_r(0), B_r^+ = B_r^+(0)$ and $Q_r^+ = Q_r^+(0)$, etc.   
We note that $Q_{r}(z_0)$ and $Q_{r}^+(z_0)$ are respectively just the ball and the upper-half ball in $\bR^{d+1}$ of  radius $r$ centered at $z_0$, and they are not parabolic cylinders as commonly used in the study of parabolic equations.  
We use these balls instead of parabolic cylinders because the equation is not invariant under the usual heat scaling. 

We next give the partial mean oscillation of the coefficients that was first introduced in \cite{MR2338417, MR2300337}.
\begin{definition} \label{a-ossi-def} 
For each $z_0 =(z'_0, x_{d0})\in \overline{\Omega}_T$, and for $\rho>0$,  we denote the partial mean oscillations of $a_{ij}$ and $c_0$ by
\begin{equation*} 
\begin{split}
a^{\#}_\rho(z_0)  & =  \max_{i, j \in \{1, 2, \ldots, d\}} \fint_{Q_\rho^+(z_0)} | a_{ij}(z) -[a_{ij}]_{\rho, z_0}(x_d)| \, dz  \\
&\qquad  + \fint_{Q_\rho^+(z_0)} |c_0(z) -[c_0]_{\rho, z_0}(x_d)|\, dz,
\end{split}
\end{equation*}
where
\begin{equation*}
[a_{ij}]_{\rho, z_0}(x_d) = 
\left\{
\begin{array}{ll}
\displaystyle{\fint_{Q_\rho'(z_0')} a_{ij}(t, x', x_d)\, dx'dt }& \quad \text{if} \quad  i \in\{1, 2,\ldots, d\}, \text{ and }  j  \not=d, \\
\displaystyle{\fint_{Q_\rho^+(z_0)} a_{id} \,dxdt }& \quad \text{for} \quad i =1, 2,\ldots, d, \text{ and }  j  =d,
\end{array} \right.
\end{equation*}
and 
\begin{equation*} 
\begin{split}
& [c_0]_{\rho, z_0}(x_d) = \fint_{Q_\rho'(z_0')} c_0(t, x', x_d) \, dx'dt. 
\end{split}
\end{equation*}
\end{definition}
We note that from the definition that the coefficients $[a_{id}]_{\rho, z_0}$ are constant for all $i = 1, 2,\ldots, d$. 
On the other hand, the other coefficients $[a_{ij}]_{\rho, z_0}$,  and  $[c_0]_{\rho, z_0}$ are functions of $x_d$-variable, for $j \not= d$, as they are correspondingly the averages of $a_{ij}, c_0$ with respect to the variable $z' = (t,x')$ only. 
We also emphasize that we do not impose any symmetry assumption on the coefficient matrix $(a_{ij})$.

We now state the main result of the paper.
\begin{theorem} \label{main-thrm} 
For given $\nu \in (0,1)$ and $p \in (1, \infty)$, there are a sufficiently large number $\lambda_0 = \lambda_0(d, \nu, p)>0$ and a sufficiently small number $\gamma = \gamma(d, \nu, p) >0$ such that the following assertions hold. 
Assume   \eqref{elli}, \eqref{a-c.conds}, and
\begin{equation} \label{a-BMO-cond}
\sup_{z_0 \in \Omega_T}\sup_{\rho \in (0, \rho_0)} a^{\#}_\rho(z_0) \leq \gamma 
\end{equation}
with some $\rho_0>0$. 
Then, for $\lambda \geq \lambda_0\rho_0^{-1}$, $f \in L_p(\Omega_T, x_d^{-p/2})$, and $F \in L_p(\Omega_T)^d$, there exists a unique weak solution $u \in \sH^{1}_{p}(\Omega_T)$ of \eqref{main-eqn}-\eqref{main-bdr-cond}. 
Moreover, 
\begin{equation} \label{main-est-0508}
 \|Du\|_{L_p(\Omega_T)} +   \sqrt{\lambda} \|u\|_{L_p(\Omega_T, x_d^{-p/2})} 
   \leq  N \Big[\|F\|_{L_p(\Omega_T)}  +   \|f\|_{L_p(\Omega_T, x_d^{-p/2})} \Big]
\end{equation}
with $N = N(\nu, d, p)>0$.
\end{theorem}

\begin{remark} The following points are worth mentioning.
\begin{itemize}
\item[\textup{(i)}] By Hardy's inequality, we have $\|u\|_{L_p(\Omega_T, x_d^{-p})} \leq N(p, d) \|Du\|_{L_p(\Omega_T)}$. 
Therefore, it follows  that Theorem \ref{main-thrm} also provides the $L_p$-estimate of $u$ with weight $x_d^{-p}$.
\item[\textup{(ii)}] It is possible to extend Theorem \ref{main-thrm} and obtain the wellposedness of \eqref{main-eqn}-\eqref{main-bdr-cond} in mixed-norm weighted space with  Muckenhoupt weights as in \cite[Theorem 2.5]{Dong-Phan-1}. 
Similarly, local boundary regularity estimates as in \cite[Corollary 2.3]{Dong-Phan-1} can be derived as an application of Theorem \ref{main-thrm}. However, we choose not to include those results to avoid further technical complications.
\item[\textup{(iii)}] It is well-known that the condition \eqref{a-BMO-cond} is necessary. We also note that in \eqref{a-BMO-cond}, the partial mean oscillations of the coefficients are measured with respect to the usual Lebesgue measure, see Definition \ref{a-ossi-def}. On the other hand,  for the classes of equations studied in \cite{CMP, Dong-Phan, Dong-Phan-1, Dong-Phan-2, Dong-Phan-3}, the mean oscillations of the coefficients are defined with respect to  suitable weights.
\end{itemize}
\end{remark}

Next, we give a quick application of Theorem \ref{main-thrm}, which could be useful later in studying the degenerate viscous Hamilton-Jacobi equation \eqref{HJ-eqn}. 
For simplicity, let us consider the model equation \eqref{simplest-eqn} with $\lambda =1$. 
The general case with $\lambda>0$ can be derived from the result with $\lambda =1$ using a scaling argument.

\begin{corollary} \label{cor-2} 
{Let $p \geq 2$, $T \in (-\infty, \infty]$, and $f_t, f \in L_p(\Omega_T, x_d^{-p/2})$. 
Then there is a unique solution $u$  of \eqref{simplest-eqn}-\eqref{main-bdr-cond} with $\lambda =1$ satisfying}
\begin{equation} \label{est-model-eqn}
\begin{split}
& \|u\|_{L_p(\Omega_T, x_d^{-p/2})}  +  \|Du\|_{L_p(\Omega_T)} +  \|u_t \|_{L_p(\Omega_T, x_d^{-p/2})} \\
& + \|D^2u\|_{L_p(\Omega_T, x_d^{p/2})} + \|Du_t\|_{L_p(\Omega_T)}  \\
& \leq N \Big[\|f_t\|_{L_p(\Omega_T, x_d^{-p/2})} + \|f\|_{L_p(\Omega_T, x_d^{-p/2})} \Big],
\end{split}
\end{equation}
where $N = N(d, p) >0$.
\end{corollary}
We emphasize that Corollary \ref{cor-2} is just for demonstrative purposes, and it is not sharp.  
More comprehensive estimates and wellposedness results for a class of non-divergence form equations of the form \eqref{simplest-eqn}-\eqref{main-bdr-cond} will be established in our forthcoming paper.
{Linear and nonlinear classes of equations with similar degenerate coefficients but more general weights compared to $x_d$ will be investigated in near future}. We note that the weights in our main results (Theorem \ref{main-thrm} and Corollary \ref{cor-2}) are different from those in \cite{Koch, Dong-Phan}.

\section{Theory of $L_2$-weak solutions} \label{sec:L2}
In this section, we study following class of parabolic equations which are slightly more general than \eqref{main-eqn}
\begin{equation} \label{L2-main-eqn}
x_d ^{-1}(a_0(x) u_t(z) + \lambda c_0(z) u(z)) - D_i\big(a_{ij}(z) D_{j} u(z) - F_i\big) = \sqrt{\lambda} x_d ^{-1} f \quad \text{in} \quad \Omega_T
\end{equation}
with the boundary condition
\begin{equation} \label{L2-main-bdr-cond}
u = 0 \quad \text{on} \quad (-\infty, T) \times \partial \bR^d_+.
\end{equation}
Here, $\lambda \geq 0$ and $F_i : \Omega_T \rightarrow \mathbb{R}$, $f: \Omega_T \rightarrow \mathbb{R}$ are given measurable functions, and $a_0 :  \bR^d_+ \rightarrow \bR$ such that
\begin{equation} \label{a.conds}
\nu \leq a_0(x) \leq \nu^{-1}, \quad  x \in \bR^d_+.
\end{equation}
We begin with the following lemma on the energy estimate for \eqref{x-d.model-eqn}. 

\begin{lemma} \label{lemma-ener-1}  
Suppose that \eqref{elli}, \eqref{a-c.conds}, and \eqref{a.conds} hold, and suppose that $f \in L_2(\Omega_T, x_d^{-1}), F \in L_2(\Omega_T)^d$. Then, for every weak solution $u \in  \sH^{1}_2(\Omega_T)$ of  \eqref{L2-main-eqn}-\eqref{L2-main-bdr-cond} with $\lambda \geq 0$, it holds that
\begin{equation} \label{1016-2.est}
\begin{split}
& \|Du\|_{L_2(\Omega_T)} + \sqrt{\lambda} \|u\|_{L_2(\Omega_T, x_d^{-1})}   \leq N \Big[ \|F\|_{L_2(\Omega_T)} + \|f\|_{L_2(\Omega_T, x_d^{-1})} \Big],
\end{split}
\end{equation}
where $N = N(\nu, d)$.
\end{lemma}

\begin{proof}  
First, by multiplying \eqref{L2-main-eqn} by $u$ and using the integration by parts with the boundary condition that $u =0$ on $\{x_d =0\}$ and the ellipticity and boundedness conditions in \eqref{elli}, we obtain
\[
\begin{split}
& \frac{1}{2} \frac{d}{dt} \int_{\bR^{d}_+} a_0(x)  |u|^2 x_d^{-1} \,dx + \lambda  \int_{\bR^{d}_+} c_0(z)  |u|^2 x_d^{-1} \,dx +  \nu \int_{\bR^{d}_+} |Du|^2 \,dx \\
& \leq  \sqrt{\lambda} \int_{\bR^{d}_+}  |u| |f| x_d^{-1} \,dx + N(\nu) \int_{\bR^{d}_+} |F| |Du|\, dx
\end{split}
\]
Integrating this estimate in the time variable, and using Young's inequality for the term on the right hand side, then cancelling similar terms, we obtain
\[
\lambda  \int_{\Omega_T} |u|^2 x_d^{-1}  \,dz + \int_{\Omega_T} |Du|^2  \,dz \leq N  \int_{\Omega_T} \big( |f(t,x)|^2 x_d^{-1} + |F(z)|^2\big)  \,dz,
\]
where we have also used conditions \eqref{a-c.conds} and \eqref{a.conds}. 
Therefore,
\begin{equation} \label{1016-3.est}
\|Du\|_{L_2(\Omega_T)} + \sqrt{\lambda} \|u\|_{L_2(\Omega_T, x_d^{-1})}  \leq N\Big[ \|F\|_{L_2(\Omega_T)} +  \|f\|_{L_2(\Omega_T, x_d^{-1})}\Big]
\end{equation}
and the lemma is proved.
\end{proof}

We now conclude this section with the following important theorem which proves Theorem \ref{main-thrm} for the case $p=2$.

\begin{theorem} \label{thm-simpl-eqn-2}  
Suppose that \eqref{elli}, \eqref{a-c.conds}, and \eqref{a.conds} hold. 
Then, for every $\lambda >0$,  $f \in L_2(\Omega_T, x_d^{-1}), F \in L_2(\Omega_T)^d$, there exists a unique weak solution $u \in \sH^{1}_2 (\Omega_T)$ of  \eqref{L2-main-eqn}-\eqref{L2-main-bdr-cond}. 
Moreover, 
\begin{equation*} 
\begin{split}
& \|Du\|_{L_2(\Omega_T)} + \sqrt{\lambda} \|u\|_{L_2(\Omega_T, x_d^{-1})}   \leq N \Big[ \|F\|_{L_2(\Omega_T)} + \|f\|_{L_2(\Omega_T, x_d^{-1})} \Big],
\end{split}
\end{equation*}
where $N = N(\nu, d)$.
\end{theorem}

\begin{proof}  
For each $n \in \bN$, let 
\begin{equation*} \label{hat-Q}
\hat{Q}_n= (-n, \min\{n, T\}) \times B_n^+.
\end{equation*}
 We consider the equation of $u$ in $\hat{Q}_n$
\begin{equation} \label{w-eqn.Qn}
 x_d^{-1} \big(a_0 u_t + \lambda c_0 u\big) -  D_i\big( \overline{a}_{ij}(z) D_{j} u -F_i\big) = x_d^{-1} f(z) \quad \text{in} \quad \hat{Q}_n
\end{equation}
with boundary condition $u=0$ on $ (-n, \min\{n, T\}) \times \partial B_n^+$ and zero initial data at $\{-n\} \times B_n^+$. 
Then, performing the energy estimates as in the proofs  of Lemma \ref{lemma-ener-1}, for each $n \in \bN$, if $u_n \in \sH^{1}_2(\hat{Q}_n)$ is a weak solution of \eqref{w-eqn.Qn}, we have the following a priori estimate
\[
\begin{split}
& \|u_n\|_{L_\infty((-n, \min\{n, T\}), L_2(B_n^+, x_d^{-1}))} + \sqrt{\lambda} \| u_n\|_{L_2(\hat{Q}_n, x_d^{-1})} + \|Du_n\|_{L_2(\hat{Q}_n)} \\
&\leq N \Big[  \|F\|_{L_2(\hat{Q}_n)} + \|f\|_{L_2(\hat{Q}_n, x_d^{-1})} \Big], 
 \end{split}
\]
for $N = N(d, \nu)>0$. 
By using this estimate and  the Galerkin method, we see that for each $n \in \bN$, there exists a unique weak solution $u_n \in \sH^{1}_2(\hat{Q}_n)$ of \eqref{w-eqn.Qn}. 
By taking $u_n =0$ in $\Omega_T \setminus \hat{Q}_n$, we can consider $u_n$ as a function defined in $\Omega_T$ satisfying
\[
\begin{split}
& \|u_n\|_{L_\infty((-\infty, T), L_2(\bR^d_+, x_d^{-1}))} + \sqrt{\lambda} \| u_n\|_{L_2(\Omega_T, x_d^{-1})} + \|Du_n\|_{L_2(\Omega_T)} \\
&\leq N \Big[ \|F\|_{L_2(\Omega_T)} + \|f\|_{L_2(\Omega_T, x_d^{-1})} \Big].
 \end{split}
\]
From this, and by passing through a subsequence, we can find $u \in \sH^1_2(\Omega_T)$ such that 
\[
\begin{split}
& u_n \rightharpoonup u  \quad \text{ in } \quad L_2(\Omega_T, x_d^{-1})  \quad \text{as} \quad n \rightarrow \infty,\\
& Du_n \rightharpoonup Du  \quad \text{ in } \quad L_2(\Omega_T)   \quad \text{as} \quad n \rightarrow \infty.
\end{split}
\]
Then, using the weak form in Definition \ref{weak-form} and passing through the limit, we see that $u \in \sH^1_2(\Omega_T)$ is a weak solution of \eqref{L2-main-eqn}-\eqref{L2-main-bdr-cond}. 
Moreover, it also holds that
\[
\begin{split}
& \|Du\|_{L_2(\Omega_T)} + \sqrt{\lambda} \|u\|_{L_2(\Omega_T, x_d^{-1})}   \leq N \Big[ \|F\|_{L_2(\Omega_T)} + \|f\|_{L_2(\Omega_T, x_d^{-1})} \Big].
\end{split}
\]
The uniqueness of $u \in \sH^{1}_2(\Omega_T)$ also follows from this estimate. 
The proof of the theorem is completed.
\end{proof}

\section{Equations with coefficients depending only on $x_d$}\label{sec:xd}

{As we apply the perturbation technique to study \eqref{main-eqn}, we need to investigate the classes of equations with coefficients depending only on $x_d$. 
Throughout the section, assume that $\overline{a}_0, \overline{c}_0 : \bR_+ \rightarrow \bR$ are measurable and they satisfy
\begin{equation} \label{a-c.cond}
\nu \leq \overline{a}_0(x_d),  \overline{c}_0(x_d) \leq \nu^{-1} \quad \text{ for } x_d \in \bR_+,
\end{equation}
for a given constant $\nu \in (0,1)$. 
Moreover,  $(\overline{a}_{ij})_{i,j=1}^d: \bR_+ \rightarrow \mathbb{R}^{d \times d}$ is a  matrix of measurable functions satisfying the following ellipticity  and boundedness conditions
\begin{equation} \label{elli-cond}
\nu|\xi|^2 \leq \overline{a}_{ij}(x_d) \xi_i \xi_j, \quad |\overline{a}_{ij}(x_d)| \leq \nu^{-1}  \quad \text{ for } x_d \in \bR_+, \end{equation}
for $\xi =(\xi_1, \xi_2, \ldots, \xi_d) \in \bR^d$.   In addition to \eqref{elli-cond}, we assume that the matrix $(\overline{a}_{ij})$ satisfies the following condition on $\overline{a}_{ij}$
\begin{equation} \label{structure.cond}
\overline{a}_{id}  = \text{constant} \quad \text{ for } i =1, 2,\ldots, d.
\end{equation}

For a fixed constant $\lambda \geq 0$,  let $\mathcal{L}_0$ be an operator in divergence form with singular coefficients
\begin{equation} \label{L-0}
\mathcal{L}_0[u] = x_d^{-1}(\overline{a}_0(x_d)u_t + \lambda \overline{c}_0(x_d) u )- D_i\big( \overline{a}_{ij}(x_d) D_{j} u\big),
\end{equation}
for $(t,x) = (t, x', x_d) \in \Omega_T$ and we study following equation  
\begin{equation} \label{x-d.model-eqn}
\left\{
\begin{array}{cccl}
\mathcal{L}_0[u]  & = &  D_i F_i + \sqrt{\lambda}  x_d^{-1} f(t,x)  & \quad \text{in} \quad \Omega_T,  \\
u  & = & 0 & \quad \text{on} \quad \{x_d =0\} .
\end{array} \right.
\end{equation}
The following theorem is the main result of the section.

\begin{theorem}\label{thm-simpl-eqn} 
Let $\nu \in (0,1), p \in (1, \infty)$ and suppose that \eqref{a-c.cond}, \eqref{elli-cond}, and \eqref{structure.cond} hold. 
Then, for every $f \in L_p(\Omega_T, x_d^{-p/2}), F = (F_1, F_2, \ldots, F_d) \in L_p(\Omega_T)^d$, and $\lambda >0$, there exists a unique  weak solution $u \in \sH^{1}_p (\Omega_T)$ of  \eqref{x-d.model-eqn}. 
Moreover, 
\begin{equation} \label{apr-est-0606}
\|Du\|_{L_p(\Omega_T)} + \sqrt{\lambda} \| u\|_{L_p(\Omega_T, x_d^{-p/2})}   \leq N\Big[ \|F\|_{L_p(\Omega_T)} + \|f\|_{L_p(\Omega_T, x_d^{-p/2})}\Big],
\end{equation}
  where $N = N(\nu, d, p) >0$.
\end{theorem}

\subsection{Pointwise estimates for homogeneous equations} \label{subsec:hom}
Let $r >0$ and  $z_0 = (t_0, x_0) \in \bR \times \overline{\bR}^d_+$. 
We consider the equation
\begin{equation} \label{eq11.52}
\left\{
\begin{array}{cccl}
\mathcal{L}_0 [u]  & = & 0 &  \quad \text{in} \quad Q_r^+(z_0),\\
 u & =& 0 & \quad \text{on} \quad Q_r(z_0) \cap \{x_d =0\} \quad \text{if} \quad Q_r(z_0) \cap \{x_d =0\} \not=\emptyset.
\end{array} \right.
\end{equation}
\begin{lemma}[Caccioppoli type estimates]  \label{Caccio-lemma} 
Let $z_0 = (z_0', 0)$ and $u\in \sH^{1}_{2}(Q_1^+(z_0))$ be a weak solution to \eqref{eq11.52} in $Q_1^+(z_0)$. 
Then, for every $0 < r <R <1$, we have
\begin{equation} \label{lemma.12-4}
\begin{split}
\int_{Q_r^+(z_0)} \Big(|Du|^2  + \lambda x_d^{-1} |u|^2 \Big)\, dz  & \leq N  \int_{Q_R^+(z_0)}  x_d^{-1}|u|^2 \, dz, 
\end{split}
 \end{equation}
where $N = N(d, \nu,r , R)>0$. 
Moreover,
\begin{equation} \label{lemma.12-4-t}
\int_{Q_r^+(z_0)}  x_d^{-1} |u_t|^2 \, dz  \leq N (d, \nu,r , R) \int_{Q_R^+(z_0)}  \Big[ |Du|^2  + \lambda x_d^{-1}|u|^2 \Big]\, dz.
\end{equation}
\end{lemma}

\begin{proof} 
By translations, it suffices to prove the lemma for the case $z_0 =0$. 
For $0 < r <R \leq 1$, using $u\varphi^2$ as a test function for \eqref{eq11.52}, where $\varphi \in C_0^\infty(Q_R)$ satisfying $\varphi \equiv 1$ in $Q_r$, we obtain 
\[
\begin{split}
&\frac{1}{2}\frac{d}{dt}\int_{B_1^+} x_d^{-1} \overline{a}_0(x_d)u^2 \varphi^2  \,dx + \lambda  \int_{B_1^+} \overline{c }_0(x_d)  x_d^{-1} u^2 \varphi^2  \,dx \\
&  + \int_{B_1^+} \overline{a}_{ij}(x_d) D_i u D_j u \varphi^2  \,dx = \int_{B_1^+}[\overline{a}_0(x_d) x_d^{-1} u^2 \varphi \varphi_t - 2 u \varphi \overline{a}_{ij}(x_d)D_j u D_i \varphi ]  \,dx.
\end{split}
\]
Then, by integrating in the time variable, using the Young's inequality, \eqref{a-c.cond}, and \eqref{elli-cond},  we obtain
\begin{equation*} 
\begin{split}
& \int_{Q_r^+} |Du|^2  \, dz + \lambda \int_{Q_r^+}x_d^{-1}|u|^2\, dz  \leq \frac{N (d, \nu)}{R-r} \int_{Q_R^+}x_d^{-1}|u|^2\, dz, 
\end{split}
\end{equation*}
where we have used the fact that $x_d^{-1} u^2 \geq u^2$ for $x_d \in (0,1)$. 
This proves  \eqref{lemma.12-4}. 

\medskip

Next,  by using the method of difference quotients, we can formally assume that $u_t$ solves the same equation as $u$ does. Therefore, applying \eqref{lemma.12-4} for $u_t$, we obtain
\begin{equation} \label{ut-1306}
\int_{Q_r^+(z_0)} \Big(|Du_t|^2  + \lambda x_d^{-1} |u_t|^2 \Big)\, dz   \leq \frac{N (d, \nu)}{R-r}  \int_{Q_R^+(z_0)}  x_d^{-1}|u_t|^2 \, dz, 
\end{equation}
Similarly, we can formally  use $u_t\varphi^2$ as a test function for \eqref{eq11.52} to yield
\[
\begin{split}
& \int_{B_1^+}x_d^{-1} \overline{a}_0(x_d) |u_t|^2 \varphi^2  \, dx + \lambda \int_{B_1^+} x_d^{-1} \overline{c}_0(x_d) u u_t \varphi^2 dx  \\
& + \int_{B_1^+} \big (\overline{a}_{ij} D_j u D_i u_t \varphi^2 +2 \varphi u_t \overline{a}_{ij} D_j u D_i \varphi \big) \, dx=0.
\end{split}
\]
Then, it follows from \eqref{a-c.cond}, \eqref{elli-cond} and Young's inequality
 that
\[
\begin{split}
& \int_{B_1^+}x_d^{-1} |u_t|^2 \varphi^2  \, dx +  \frac{\lambda}{2} \frac{d}{dt} \int_{B_1^+} x_d^{-1} \overline{c}_0(x_d) |u|^2 \varphi^2 dx  \\
& \leq N (d, \nu)\int_{B_1^+} \big[ \lambda x_d^{-1}  |u|^2 |\varphi_t| + |D_j u| |D_i u_t| \varphi +  |u_t| |D_j u| |D_i \varphi| \big]\varphi  \, dx \\
& \leq \frac{1}{2} \int_{B_1^+}x_d^{-1} |u_t|^2 \varphi^2  \, dx  \\
& \qquad +  N (d, \nu) \int_{B_1^+} \big[|Du|^2(\epsilon^{-1} \varphi + x_d|D\varphi|^2) + \epsilon |Du_t|^2 \varphi  + \lambda x_d^{-1} |u|^2 |\varphi_t| \varphi  \big]\, dx,
\end{split}
\]
for $\epsilon \in (0,1)$. Then integrating this equality in time variable, we obtain
\[
\begin{split}
&\int_{Q_r^+}x_d^{-1} |u_t|^2  \, dz \\
&\leq  N (d, \nu) \int_{Q_R^+} \big[ ( \epsilon^{-1} + (R-r)^{-2}) |Du|^2 + \lambda  x_d^{-1} (R-r)^{-1} |u|^2 + \epsilon |Du_t|^2 \big]\, dz.
\end{split}
\]
Using this, and \eqref{ut-1306} with slightly different balls, we conclude that
\[
\int_{Q_r^+}x_d^{-1} |u_t|^2  \, dz \leq  \frac{N}{(R-r)^2}  \int_{Q_R^+} \big[\epsilon^{-1}|Du|^2 + \lambda  x_d^{-1} |u|^2 \big]\, dz +  \frac{\epsilon N}{R-r} \int_{Q_R^+} x_d^{-1}|u_t|^2\, dz,
\]
for all  $0< r <R <1$, $N = N(d,\nu)>0$, and for all $\epsilon \in (0,1)$. 
From this, and by the standard iteration argument (see \cite[Lemma 5.13, p. 82]{Gia-Mar} for example), we conclude that
\[
\int_{Q_r^+}x_d^{-1} |u_t|^2  \, dz \leq N(d,\nu, r, R)  \int_{Q_R^+} \big[ |Du|^2 + \lambda  x_d^{-1} |u|^2 \big]\, dz.
\]
 The estimate \eqref{lemma.12-4-t} is proved and the proof of the lemma is completed.
\end{proof}

Next, we state and prove an important estimate for $w(z)= x_d^{-1} u(z)$ for which the assumption $\overline{a}_{id} =$ constant in \eqref{structure.cond} is needed.

\begin{lemma} \label{w-lemma-0529}
Let $z_0 = (z_0', 0)$ and $u\in \sH^{1}_{2}(Q_1^+(z_0))$ be a weak solution to \eqref{eq11.52} in $Q_1^+(z_0)$. 
Then, for every $0 < r <R$,
\[
\int_{Q_{r}^+(z_0)} \big[x_d |Dw|^2 + \lambda w^2 \big]  \,dz \leq N\int_{Q_R^+(z_0)} |w|^2  \,dz,
\]
where $w(z) =x_d^{-1} u(z)$ for $z=(z', x_d) \in Q_1^+$, and $N = N(r, R, d, \nu)>0$.
\end{lemma}

\begin{proof} 
We can assume $z_0' =0$. 
For $\epsilon >0$ and sufficiently small, let $w_\epsilon(z) = (x_d +\epsilon)^{-1} u(z)$. 
Observe that 
\begin{equation} \label{D-w-epsi}
\begin{split}
& D_i w_\epsilon(z) = (x_d+\epsilon)^{-1} D_i u(z) - \delta_{id}(x_d+\epsilon)^{-2} u(z), \\
& \partial_t w_\epsilon =  (x_d+\epsilon)^{-1}\partial_t u(z),
\end{split}
\end{equation}
and moreover $w_\epsilon \in \sH^{1}_2(Q_1^+)$. 
We  then can test the equation of $u$ with $w_{\epsilon} \varphi^2$ to obtain
\[
\begin{split}
& \frac{1}{2}\frac{d}{dt}\int_{B_1^+} \eta_{\epsilon}(x_d)\overline{a}_0(x_d) w_{\epsilon}^2 \varphi^2  \, dx + \lambda \int_{B_1^+} \eta_{\epsilon}(x_d)\overline{c}_0(x_d) w_{\epsilon}^2\varphi^2  \, dx \\
& +  \int_{B_1^+} \overline{a}_{ij} D_j ((x_d +\epsilon) w_{\epsilon}) D_i w_{\epsilon} \varphi^2  \, dx \\
& = \int_{B_1^+}\big(\eta_\epsilon(x_d)\overline{a}_0(x_d) w_{\epsilon}^2 \varphi \varphi_t  - 2  w_\epsilon \varphi \overline{a}_{ij} D_j ((x_d+\epsilon)w_\epsilon) D_i\varphi \big)  \,dx,
\end{split}
\]
where $\eta_\epsilon(x_d) = \frac{x_d+\epsilon}{x_d}$. 
Since, $D_j((x_d +\epsilon) w_\epsilon) = (x_d +\epsilon) D_j w_\epsilon  + \delta_{jd} w_\epsilon$,  it follows  that
\[
\begin{split}
& \frac{1}{2}\frac{d}{dt}\int_{B_1^+} \eta_\epsilon(x_d) \overline{a}_0(x_d) w_\epsilon^2 \varphi^2 \, dx + \lambda \int_{B_1^+} \eta_\epsilon(x_d) \overline{c}_0(x_d) w_\epsilon^2\varphi^2  \, dx \\
& + \int_{B_1^+} (x_d+\epsilon) \overline{a}_{ij} D_j w_{\epsilon} D_i w_{\epsilon} \varphi^2  \, dx \\
& = \int_{B_1^+}\Big(\eta_\epsilon \overline{a}_0(x_d) w_\epsilon^2 \varphi \varphi_t  - 2 (x_d +\epsilon) w_\epsilon \varphi \overline{a}_{ij} D_j w_\epsilon D_i\varphi \\
& \qquad \qquad - 2 w_\epsilon^2 \varphi \overline{a}_{id}D_i\varphi  - \frac{1}{2}  D_i \big[\overline{a}_{id} w^2_\epsilon\big] \varphi^2 \Big) \, dx \\
& = \int_{B_1^+}\big(\eta_\epsilon \overline{a}_0(x_d) w_\epsilon^2 \varphi \varphi_t  - 2 (x_d+\epsilon) w_\epsilon \varphi \overline{a}_{ij} D_j w_{\epsilon} D_i\varphi - w_\epsilon^2 \varphi \overline{a}_{id}D_i\varphi \big)  \,dx,
\end{split}
\]
where we have used the fact that $\overline{a}_{id}$ are constant assumed in \eqref{structure.cond}  in the middle step of the above calculation to write $w_\epsilon \overline{a}_{id} D_i w_\epsilon = \frac{1}{2} D_i[\overline{a}_{id} w_\epsilon^2]$, and the integration by parts on this term in the last step. 
Now, by \eqref{elli-cond}, \eqref{a-c.cond}, and the fact that $\eta_\epsilon \geq 1$, it follows that
\[
\begin{split}
& \frac{1}{2}\frac{d}{dt}\int_{B_1^+} \eta_\epsilon(x_d) \overline{a}_0(x_d) w_\epsilon^2 \varphi^2  \, dx + \lambda \int_{B_1^+} w_\epsilon^2\varphi^2 \, dx  + \int_{B_1^+} (x_d+\epsilon) |D w_{\epsilon}|^2 \varphi^2  \, dx \\
&\leq N(d, \nu)\int_{B_1^+}\big(\eta_\epsilon w_\epsilon^2 |\varphi \varphi_t| +  (x_d+\epsilon) |w_\epsilon| \varphi |D w_{\epsilon}| |D\varphi| + w_\epsilon^2 |\varphi|  |D\varphi| \big) \,  dx \\
& \leq \frac{1}{2}  \int_{B_1^+} (x_d+\epsilon) |D w_{\epsilon}|^2 \varphi^2  \, dx \\
& \qquad + N(\nu, d) \int_{B_1^+}\big(\eta_\epsilon w_\epsilon^2 |\varphi \varphi_t| + (x_d+\epsilon) |w_\epsilon|^2 |D \varphi|^2 + w_\epsilon^2 |\varphi | |D\varphi|\big)  \, dx,
\end{split}
\]
where we applied Young's inequality in the last step.
 Then,
\[
\begin{split}
& \frac{1}{2}\frac{d}{dt}\int_{B_1^+} \eta_\epsilon(x_d) \overline{a}_0(x_d) w_\epsilon^2 \varphi^2  \, dx + \lambda \int_{B_1^+} w_\epsilon^2\varphi^2 \, dx  + \int_{B_1^+} (x_d+\epsilon) |D w_{\epsilon}|^2 \varphi^2  \, dx \\
& \leq  N(\nu, d) \int_{B_1^+} w_\epsilon^2 \big(\eta_\epsilon  \varphi |\varphi_t| +  |D \varphi|^2 +  |D\varphi|\big)  \,dx.
\end{split}
\]
Therefore,
\[
\begin{split}
& \lambda \int_{Q_1^+} w_\epsilon^2\varphi^2  \, dz  + \int_{Q_1^+} (x_d+\epsilon) |D w_{\epsilon}|^2 \varphi^2 \, dz\\
& \leq  N(\nu, d) \int_{Q_1^+} w_\epsilon^2 \big(\eta_\epsilon  \varphi |\varphi_t| +  |D \varphi|^2 +  |D\varphi|\big) \, dz \\
& \leq  N(\nu, d, R, r) \int_{Q_R^+} w^2  \,dz <\infty,
\end{split}
\]
where the assertion on the finiteness of the above integral is due to Hardy's inequality. 
Finally, from the explicit formula of $w_\epsilon$,  \eqref{D-w-epsi},  and  by Fatou's lemma, we obtain by passing $\epsilon \rightarrow 0^+$ that
\[
\int_{Q_{r}^+(z_0)} \big[x_d |Dw|^2 + \lambda w^2 \big]  \,dz \leq N(r, R, d, \nu) \int_{Q_R^+(z_0)} |w|^2 \, dz.
\]
The proof of the lemma is now complete.
\end{proof}

\begin{remark} \label{Hardy-remark} 
By Hardy's inequality, we have
\[
\int_{Q_R^+(z_0)} |w|^2 \, dz \leq N \int_{Q_R^+(z_0)} |Du|^2 \, dz.
\]
Therefore, the estimates in Lemma \ref{w-lemma-0529} are well-defined for $u \in \sH^1_2(Q_1^+(z_0))$.
\end{remark}

Next, we derive the following Caccioppoli type estimates on higher derivatives of solutions.

\begin{lemma} \label{higher-Caccio}
 Let $z_0 = (z_0', 0)$ and let $u\in \sH^{1}_{2}(Q_1^+(z_0))$ be a weak solution to \eqref{eq11.52} in $Q_1^+(z_0)$. 
 Then, for every $j, k \in \mathbb{N}\cup\{0\}$
\begin{align} \nonumber
& \int_{Q_r ^+(z_0)} x_d \big( |\partial_t^{j+1}D^{k}_{x'} w|^2 + |DD_{x'}^kw|^2\big) \, dz + \int_{Q_r ^+(z_0)} |DD_{x'}^k \partial_t^ju|^2 \, dz \\ \nonumber
& \qquad + \lambda \int_{Q_r ^+(z_0)} x_d|D_{x'}^k \partial_t^j w|^2 \, dz \\ \label{inter-eq3.50}
& \le N\int_{Q_{R}^+(z_0)}  x_d |w(z)|^2\, dz,
\end{align}
for every $0 < r < R <1$, where $w(z) = x_d^{-1} u(z)$ for $z = (z', x_d) \in Q_1^{+}(z_0)$, and $N = N(\nu, d, k, j, r, R) >0$.
\end{lemma}

\begin{proof} 
By using translations, it suffices to prove the lemma for the case $z_0' =0$.
 The main idea is to apply Lemmas \ref{Caccio-lemma}, \ref{w-lemma-0529} and Hardy's inequality.
  We first prove \eqref{inter-eq3.50} for $j=0$ and $k \in \mathbb{N} \cup \{0\}$. 
  We observe that from Lemma \ref{Caccio-lemma} and Lemma  \ref{w-lemma-0529}, we have
\begin{align*} \nonumber
& \int_{Q_r ^+} x_d \big( |\partial_t w|^2 + |Dw|^2\big) \, dz + \int_{Q_r ^+} |D u|^2 \, dz  + \lambda \int_{Q_r ^+} x_d|w|^2 \, dz \\
& \le N(\nu, k, j, r, R) \int_{Q_{R'}^+} |w(z)|^2\, dz,
\end{align*}
where we have used the fact that $x_d \in (0,1)$ and $R' = (r+R)/2$. 
Recall that by Hardy's inequality
\[
\int_{Q_{R'}^+} |w|^2  \,dz = \int_{Q_{R'}^+} |u/x_d|^2  \,dz \leq N(d, R') \int_{Q_{R'}^+} |Du|^2  \,dz.
\]
From this and \eqref{lemma.12-4}, we see that 
\[
\int_{Q_{R'}^+} |w|^2  \,dz \leq N \int_{Q_R^+} x_d |w|^2  \,dz.
\]
Then, \eqref{inter-eq3.50} when $j=0$ and $k=0$ follows. 
We next prove \eqref{inter-eq3.50} for $j=0$ and $k=1$. 
By using the difference quotient method in $x'$-direction, we can formally assume that $D_{x'}u$ solves the same equation and satisfies the same boundary condition as $u$ does. 
Then, we apply  \eqref{inter-eq3.50} with $j=0$ and $k=0$ that we just proved for $D_{x'} u$ to infer that
\[
\begin{split}
& \int_{Q_r ^+} x_d \big( |\partial_t D_{x'}w|^2 + |DD_{x'}w|^2\big) \, dz + \int_{Q_r^+} \big(|D D_{x'} u|^2 + \lambda x_d |D_{x'}w|^2 \big)  \,dz \\
&  \leq N(\nu, r, R) \int_{Q_{R'}^+} x_d |D_{x'}w|^2 \, dz,
\end{split}
\]
Then, by combining this estimate with the estimate   \eqref{inter-eq3.50} with $j=0$ and $k=0$ that we just proved, we have
%
%
\[
\begin{split}
& \int_{Q_r ^+} x_d \big( |\partial_t D_{x'}w|^2 + |DD_{x'}w|^2 \big)\, dz+ \int_{Q_r^+} \big( |D D_{x'} u|^2 + \lambda x_d |D_{x'}w|^2 \big)  \,dz \\
&  \leq N(\nu, r, R) \int_{Q_R^+} x_d |w|^2  \,dz.
\end{split}
\]
Hence, \eqref{inter-eq3.50} is proved when $k=1$ and $j=0$. 
We can repeat the argument together with an induction on $k$ to obtain \eqref{inter-eq3.50} for $j=0$ and $k \in \mathbb{N} \cup\{0\}$.

\medskip

We now prove \eqref{inter-eq3.50} for general $j, k \in \mathbb{N} \cup \{0\}$. 
The argument is similar as before. 
Note that we have proved
\[
\begin{split}
& \int_{Q_r ^+} x_d \big( |\partial_t D_{x'}^kw|^2 + |DD_{x'}^kw\big)\, dz+ \int_{Q_r^+} \big( |D D_{x'}^k u|^2 + \lambda x_d |D_{x'}^ku|^2 \big) \, dz \\
& \leq N(\nu, r, R) \int_{Q_R^+} x_d |w|^2  \,dz,
\end{split}
\]
for all $k \in \mathbb{N} \cup \{0\}$ and $0 < r < R <1$. 
Again, by using the difference quotient method with respect to the time variable direction, we can formally assume that $u_t$ solves the same equation as $u$. 
Then, we obtain
\[
\begin{split}
& \int_{Q_r ^+} x_d |\big( \partial_t^2 D_{x'}^kw|^2 +|DD_{x'}^k \partial_t w|^2\big)\, dz + \int_{Q_r^+} \big( |D D_{x'}^k \partial_t u|^2 + \lambda x_d |D_{x'}^k \partial_t w|^2 \big)  \,dz \\
& \leq N(\nu, r, R) \int_{Q_{R'}^+} x_d |\partial_t w|^2  \,dz.
\end{split}
\]
From this, \eqref{lemma.12-4-t} and \eqref{lemma.12-4}, we infer that
\[
\begin{split}
& \int_{Q_r ^+} x_d |\partial_t^2 D_{x'}^kw|^2 + |DD_{x'}^k\partial_t w|^2 \big) \, dz + \int_{Q_r^+} \big( |D D_{x'}^k \partial_t u|^2 + \lambda x_d |D_{x'}^k \partial_t w|^2 \big)  \,dz \\
& \leq N(\nu, r, R) \int_{Q_{R}^+} x_d |w|^2  \,dz,
\end{split}
\]
which proves  \eqref{inter-eq3.50} for $j=1$ and $k \in \mathbb{N} \cup \{0\}$. 
The general case with $j \in \mathbb{N} \cup\{0\}$ can be proved similarly with an induction argument.  
The proof of the lemma is completed.
\end{proof}

Next, we give a corollary of Lemma \ref{higher-Caccio}.

\begin{corollary} \label{step-1} Under the assumptions of Lemma \ref{higher-Caccio}, we have
\begin{align} \label{0727.eqn} 
  \|D_{x'}^l \partial_t^j v\|_{Q_{1/2}^+} & \le N(d, \nu, l, j, l_0, j_0)  \|D_{x'}^{l_0} \partial_t^{j_0}v\|_{L_2(Q_{3/4}^+)}, 
\end{align}
for $j, l, j_0, l_0 \in \mathbb{N}\cup\{0\}$ with $j_0 \leq j, l_0\leq l$ and for $v(z) = x_d^{-1/2} u(z)$ with $z = (z', x_d) \in Q_1^+(z_0)$.
\end{corollary}

\begin{proof} 
Applying the Sobolev embedding theorem in the variable $z' \in \bR \times \bR^{d-1}$, we have
\begin{align*}
 |D_d D_{x'}^{l}\partial_t^j u(z',x_d)|\le  N(d) \|D_dD_{x'}^l \partial_t^j  u(\cdot, x_d)\|_{W^{k,k}_2(Q_{1/2}')},
\end{align*}
for any $z=(z',x_d)\in Q_{1/2}^+$, where $k> d/2$ is an even integer,  $Q_{1/2}'$ is the ball with respect to the variable $z' \in \bR^d$ centered at the origin.  Here, $W^{k,k}_2(Q_{1/2}')$ denotes the usual Sobolev space in which $k$ is the order of the derivative with norm
\[
\|f\|_{W^{k,k}_2(Q_{1/2}')} = \sum_{l+j \leq k} \|D_{x'}^l \partial_t^j  f\|_{L_p(Q_{1/2}')}.
\]
Then, using Lemma \ref{higher-Caccio} for $D_{x'}^{l_0} \partial_t^{j_0} u$,  we have, for $z'\in Q_{1/2}'$,
\begin{align*}  
\int_{0}^{1/2} |D_d D_{x'}^l \partial_t^j  u(z',x_d)|^2 \, dx_d
&\le N \int_{0}^{1/2} \|D_d D_{x'}^l \partial_t^j  u(\cdot, x_d)\|^2_{W^{k/2,k}_2(Q_{1/2}')}\,dx_d\nonumber\\
&\le N\|D_{x'}^{l_0} \partial_t^{j_0} u\|_{L_2(Q_{3/4}^+, x_d^{-1})}^2.
\end{align*}
From this, and  by H\"older's inequality, we infer that, for $x_d \in (0,1/2)$,
\begin{align} \nonumber 
&\int_{0}^{x_d} |D_d D_{x'}^l \partial_t^j  u(z', s)|\, ds \le \left(\int_{0}^{1/2}|D_d D_{x'}^l \partial_t^j u(z', s)|^2 \, ds\right)^{1/2} \left(
\int_{0}^{x_d}\,ds \right)^{1/2}\\  \label{eq4.16bb}
& \le N x_d^{1/2} \|D_{x'}^{l_0} \partial_t^{j_0}v\|_{L_2(Q_{3/4}^+)}.
\end{align}
Then, by the fundamental theorem of calculus and the boundary condition that $D_{x'}^l \partial_t^j u(x', 0)=0$, we infer  that
\begin{equation*} \label{0727.eqn1}
|D_{x'}^l \partial_t^j u(z', x_d)| \leq N x_d^{1/2} \|D_{x'}^{l_0} \partial_t^{j_0} v\|_{L_2(Q_{3/4}^+)}, 
\end{equation*}
for  $(z', x_d) \in Q_{1/2}^+$, and \eqref{0727.eqn} is proved. 
\end{proof}

We next state and prove the interior pointwise estimates for solutions and its spatial derivatives.
\begin{lemma} \label{interior-Linf} 
Let $z_0 = (t_0, x_0) \in \Omega_T$ and suppose that $B_{2r}(x_0) \subset \bR^d_+$. 
If $u \in \sH^{1}_{2}(Q_{2r}(z_0))$ is a weak solution to \eqref{eq11.52}, then we have
\[
\sup_{z \in Q_{r}(z_0)} |x_d^{-1/2} u(z)| \leq N \left( \fint_{Q_{2r}(z_0)} |x_d^{-1/2} u(z)|^2 \, dz \right)^{1/2} 
\]
and
 \begin{align*} 
     \|Du\|_{L_\infty(Q_r(z_0))}  \leq N \left(\fint_{Q_{2r}(z_0)}\big( |Du|^2 + \lambda |x_d^{-1/2}  u(z)|^2 \big) \, dz \right)^{1/2},
         \end{align*}
for  $N = N(\nu,  d)>0$.
\end{lemma}

\begin{proof}  
 We write $x_0 = (x_0', x_{0d}) \in \bR^{d-1} \times \bR_+$. 
 By scaling and without loss of generality, we assume $r =1$. 
 As $x_{0d} \geq 2$, the coefficients $x_d^{-1}$ is bounded  in $Q_{3/4}(z_0)$.  
 Then, we can follow the  standard regularity estimates for parabolic equations with uniformly elliptic coefficients (e.g., \cite[Lemma 3.5]{DK11}) to obtain
\[
 \sup_{z \in Q_1^+(z_0)}|x_d^{-1/2}u(z)| \le N\Big(\fint_{Q_{2}(z_0)}|x_d^{-1/2}u(z)|^2 \, dz \Big)^{1/2}.
\]
and
\[
 \|Du\|_{L_\infty(Q_{1}(z_0))}\le N\Big(\fint_{Q_{2}(z_0)}\big( |Du(z)|^2 + \lambda |x_d^{-1/2} u(z)|^2 \big) \, dz \Big)^{1/2}.
\]
\end{proof}
Our next result is the local boundary Lipschitz estimates of solutions to \eqref{eq11.52}.
\begin{lemma} \label{u-x-d.bound} 
Let $z_0 = (z', 0)$ and $u\in \sH^{1}_{2}(Q_{2r}^+(z_0))$ be a weak solution to \eqref{eq11.52} in $Q_1^+(z_0)$. 
Then
\begin{equation} \label{Dd-u.est0603}
\|Du\|_{L_\infty(Q_{r}^+(z_0))} \leq N \big( \|Du\|_{L_2(Q_{2r}^+(z_0))} + \sqrt{\lambda} \|u\|_{L_2(Q_{2r}^+(z_0),x_d^{-1})} \big),
\end{equation}
where $N = N(\nu, d) >0$. 
\end{lemma}

\begin{proof}  
By translation and dilation, we assume that $z_0' =0$ and $r=1/2$. 
It follows from Corollary \ref{step-1} and Hardy's inequality that 
\begin{equation} \label{Dux'-0606}
\sup_{z \in Q_{1/2}^+ }|x_d^{-1/2}D_{x'} u| \leq N \|Du\|_{L_{2}(Q_1^+)}.
\end{equation}
Therefore, we only need to control $D_du$. Let us denote $v(z) = x_d^{-1/2} u(z)$ with $z = (z', x_d) \in Q_{1}^+$. Using  Corollary \ref{step-1} again, we see
\begin{equation}  \label{v-boun--05.27}
\begin{split}
&  \|v\|_{L_\infty(Q_{1/2}^+)} \leq N \|v\|_{L_2(Q_{3/4}^+)} \quad \text{ and} \\
& \|D_{x'}^2 v\|_{L_\infty(Q_{1/2}^+)} + \|D_{x'} v\|_{L_\infty(Q_{1/2}^+)} \leq N \|v\|_{L_2(Q_{3/4}^+)}.
\end{split}
\end{equation}
Also, by combining Corollary \ref{step-1}  and Lemma \ref{Caccio-lemma} with slightly different balls, we obtain
\[
\|v_t\|_{L_\infty(Q_{1/2}^+)} \leq  N \big( \|Du\|_{L_2(Q_1^+)} + \sqrt{\lambda} \|v\|_{L_2(Q_1^+)} \big).
\]
From this, \eqref{v-boun--05.27}, and Lemma \ref{Caccio-lemma}, it follows that
\begin{equation} \label{lambda-v}
 \lambda \|v\|_{L_\infty(Q_{1/2}^+)} + \|v_t\|_{L_\infty(Q_{1/2}^+)}  \leq  N \big( \|Du\|_{L_2(Q_1^+)} + \sqrt{\lambda} \|v\|_{L_2(Q_1^+)} \big).
\end{equation}
On the other hand, for each $\hat{z} = (z', 1/2)$ with $z' \in Q_{1/2}'$, it follows from Lemma \ref{interior-Linf} that
\begin{align}\notag
|Du(\hat{z})|  & \leq N \left(\fint_{Q_{1/100}(\hat{z})} |Du(z)|^2  + \lambda |v(z)|^2 \,dz \right)^{1/2} \\  \label{Dhatz-0531.est}
&  \leq N\big[ \|Du\|_{L_2(Q_1^+)} + \sqrt{\lambda} \|v\|_{L_2(Q_1^+)} \big].
\end{align}
Next, it follows from the estimate \eqref{eq4.16bb} with $l=1, j=0, l_0 = j_0 =0$ and Hardy's inequality that
\begin{equation} \label{in-DDx'-0603}
\int_{0}^{1/2} |DD_{x'}u(z', s)|\, ds \leq N\|Du\|_{L_{2}(Q_1^+)}, \quad \forall  z' \in Q_{1/2}'.
\end{equation}
Now, let us denote $\cU = \overline{a}_{di}(x_d)D_i u$. Under the assumptions \eqref{a-c.cond}, \eqref{elli-cond}, it follows from the PDE of $u$ in \eqref{eq11.52} that
\begin{align} \notag 
|D_d\, \cU(z)| & \leq N(d) \big[x_d^{-1/2} \big(|v_t(z)| + \lambda |v(z)|\big) +  {|DD_{x'} u|} \big] \\ \label{Ddd.0603.est}
& \leq  N x_d^{-1/2}\big( \|Du\|_{L_2(Q_1^+)} + \sqrt{\lambda} \|u\|_{L_2(Q_1^+, x_d^{-1})} \big) + {|DD_{x'} u|} .
\end{align}
Then, for each $z= (z', x_d) \in Q_{1/2}^+$, by integrating \eqref{Ddd.0603.est} on $(x_d, 1/2)$ we obtain
\begin{align*} \notag
|\cU(z', x_d)| & \leq |\cU(z', 1/2)|  + \big( \|Du\|_{L_2(Q_1^+)} + \sqrt{\lambda} \|u\|_{L_2(Q_1^+,x_d^{-1})} \big) \int_{x_d}^{1/2} s^{-1/2} ds\\ \notag
& \quad   + {\int_{x_d}^{1/2} |DD_{x'}(z', s)| ds} \label{PDE-D_du} \\
& \leq N \big( \|Du\|_{L_2(Q_1^+)} + \sqrt{\lambda} \|u\|_{L_2(Q_1^+, x_d^{-1})} \big),
\end{align*}
where we used \eqref{Dhatz-0531.est} {and \eqref{in-DDx'-0603}} in the last estimate. 
This estimate, \eqref{Dux'-0606}, the ellipticity condition in \eqref{elli-cond}, and the definition of $\cU$ imply \eqref{Dd-u.est0603}. 
The proof is completed.
\end{proof}

From Corollary \ref{step-1}, Lemma \ref{interior-Linf}, and Lemma \ref{u-x-d.bound}, we obtain the following Corollary on the solution decomposition which will be used in the proof of Theorem \ref{thm-simpl-eqn}.

\begin{corollary} \label{Simple-approx} 
Let $z_0\in \overline{\Omega}_T$ and $r >0$. 
Suppose that $F \in L_{2}(Q_{10r}^+(z_0))^d$, $f \in L_{2}(Q_{10r}^+(z_0), x_d^{-1})$, and $u \in \sH^{1}_2(Q_{10r}^+(z_0))$ is a weak solution of \eqref{eq11.52} in $Q_{10r}^+(z_0)$. 
Then we can write
\[
u(t, x) = g(t, x) + h(t, x) \quad  \text{in}\,\, Q_{10r}^+(z_0),
\]
where $g$ and $h$ are functions in $\sH_2^1(Q_{10r}^+(z_0))$ and satisfy
\begin{equation}  \label{0506-tU-est}
\fint_{Q_{2r}^+(z_0)}|G(z)|^2  \, dz  \leq N \fint_{Q_{10r}^+(z_0)}\Big( |F(z)|^2 +  x_d^{-1}|f(z)|^2) \, dz
\end{equation}
and
\begin{align} \notag
\|H\|_{L_\infty(Q_{r}^+(z_0))}^{2}
& \leq N  \fint_{Q_{10r}^+(z_0)} |U|^2 \, dz \\ \label{0506-W.est}
& \quad + N \fint_{Q_{10r}^+(z_0)} \Big(|F(z)|^2 + x_d^{-1}|f(z)|^2\Big) \, dz,
\end{align}
 where $N = N(d,\nu)>0$ and
\[
\begin{split}
& G(z)=|Dg(z)|+\lambda^{1/2}|x_d^{-1/2}g(z)|,\quad H(z)=|Dh(z)|+\lambda^{1/2}|x_d^{-1/2} h(z)|, \\
& U(z)=|D u(z)|+\lambda^{1/2}|x_d^{-1/2}u(z)|,\quad \text{for } z = (z', x_d) \in Q_{10r}^+(z_0).
\end{split}
\]
\end{corollary}

\begin{proof}  
We write $z_0 = (z_0', x_{d0})$ and we split the proof into two cases.

\medskip

\noindent
{\bf Case I.} Consider $x_{0d} < 2r$. Let $\hat{z}_0 = (z_0', 0)$.
Let $g \in \sH^1_2(\Omega_T)$ be a weak solution of the equation
\begin{equation*}
\begin{split}
& x_d^{-1}  ( {\overline{a}_0(x_d)}g_t  +\lambda {\overline{c}_0(x_d)} g)  - D_i
\big(\overline{a}_{ij}(x_d) D_j g - F_{i}(z)\chi_{Q_{8r}^+(\hat{z}_0)}(z)\big)  \\
& =  \lambda^{1/2} x_d^{-1} f (z) \chi_{Q_{8r}^+(\hat{z}_0)}(z)  \quad \text{in} \  \Omega_T
\end{split}
\end{equation*}
with the boundary condition $g =0$ on $\{x_d =0\}$. 
The existence of $g$ follows from Theorem \ref{thm-simpl-eqn-2}. 
Moreover, we also have
\begin{equation} \label{G-L2.0606}
\fint_{Q_{8r}^+(\hat{z}_0)} |G(z)|^2\, dz \leq N \fint_{Q_{8r}^+(\hat{z}_0)}\big(|F(z)|^2 + x_d^{-1} |f(z)|^2 \big)\, dz.
\end{equation}
Now, as $Q_{2r}^+(z_0) \subset Q_{8r}^+(\hat{z}_0) \subset Q_{10r}^+(z_0)$, \eqref{0506-tU-est} follows from \eqref{G-L2.0606}. 
On the other hand,  let $h = u - g$. 
We see that $h \in \sH_2^1(Q_{8r}^+(\hat{z}_0))$ is a weak solution of
\[
x_d^{-1} (\overline{a}_0(x_d) h_t +\lambda \overline{c}_0(x_d) h)   - D_i \big(\overline{a}_{ij}(x_d) D_j  h\big)   = 0 \quad \text{in} \  Q_{8r}^+(\hat{z}_0)
\]
with the boundary condition $h =0$ on $Q_{8r}(\hat{z}_0) \cap \{x_d =0\}$. 
Then, by Corollary \ref{step-1} and Lemma \ref{u-x-d.bound} and the triangle inequality, we get 
\[
\begin{split}
\|H\|_{L_\infty(Q_{4r}^+(\hat{z}_0))} & \leq N \left(\fint_{Q_{8r}^+(\hat{z})} |H(z)|^2\, dz\right)^{1/2} \\
& \leq N \left(\fint_{Q_{8r}^+(\hat{z})} |U(z)|^2\, dz\right)^{1/2} +  N \left(\fint_{Q_{8r}^+(\hat{z})} |G(z)|^2\, dz\right)^{1/2}.
\end{split}
\]
From this and \eqref{G-L2.0606}, we obtain
\[
\begin{split}
\|H\|_{L_\infty(Q_{4r}^+(\hat{z}_0))}  & \leq N \left(\fint_{Q_{8r}^+(\hat{z})} |U(z)|^2\, dz\right)^{1/2} \\
& \quad + N \fint_{Q_{8r}^+(\hat{z}_0)}\big(|F(z)|^2 + x_d^{-1} |f(z)|^2 \big)\, dz.
\end{split}
\]
As $Q_{2r}^+(z_0) \subset Q_{4r}^+(\hat{z}_0) \subset Q_{8r}^+(\hat{z}_0) \subset Q_{10r}^+(z_0)$,  \eqref{0506-W.est} follows.

\medskip

\noindent
{\bf Case II.} Consider $x_{d0} >2r$. 
We use the same strategy as in {Case I} but it is simpler.
 We directly use the ball $Q_{2r}^+(z_0)$ instead of $Q_{8r}^+(\hat{z}_0)$ and Lemma \ref{interior-Linf}. 
 We skip the details.
\end{proof}

\subsection{Proof of Theorem \ref{thm-simpl-eqn}} 
We prove Theorem \ref{thm-simpl-eqn} in this subsection.

\begin{proof}
When $p =2$, Theorem \ref{thm-simpl-eqn} follows from Theorem \ref{thm-simpl-eqn-2}. 
Therefore, we only need to consider the cases when $p \in (2,\infty)$ and $p \in (1,2)$.

\medskip

\noindent
{\bf Case I.} Consider $p \in (2,\infty)$. 
We first prove the a priori estimate \eqref{apr-est-0606}. 
Let $u \in \sH_{2,\textup{loc}}^1(\Omega_T)$ be a weak solution of \eqref{x-d.model-eqn}.  
By Corollary \ref{Simple-approx} that for every $z_0 \in \overline{\Omega}_T$ and $r >0$, we have
\[
u =g + h \quad \text{in} \  Q_{2r}^+(z_0),
\]
where $g$ and $h$ satisfy \eqref{0506-tU-est} and \eqref{0506-W.est}. 
Then the estimate \eqref{apr-est-0606} follows from the standard real variable argument (see \cite{DK11b} for example). 
We omit the details.

Note also that the uniqueness of solutions follows \eqref{apr-est-0606}. 
Hence, it remains to prove the existence of the solution. For each $k \in \bN$, let
\begin{equation} \label{hat-Q}
\hat{Q}_k:= (-k, \min\{k, T\}) \times B_k^+.
\end{equation}
Also, let
\[ F^{(k)}=F(z) \chi_{\widehat Q_k}(z), \quad \text{and} \quad f^{(k)} = f(z)\chi_{\widehat Q_k}(z).\] 
As $\hat{Q}_k$ is compact, we see that $F^{(k)}\in L_2(\Omega_T)^{d}\cap L_p(\Omega_T)^{d}$. Moreover, by the dominated convergence theorem, $F^{(k)}\to F$ in $L_p(\Omega_T)^d$ as $k\to \infty$. Similarly,  $\{f^{(k)}\}\subset L_2(\Omega_T, x_d^{-1})\cap L_p(\Omega_T, x_d^{-p/2})$ and $f^{(k)}\to f$ in $L_p(\Omega_T, x_d^{-p/2})$ as $k\to \infty$. 
Now, let $u^{(k)} {\in \sH_2^1(\Omega_T)}$ be the weak solution of the equation \eqref{x-d.model-eqn} with $F^{(k)}$ and $ f^{(k)}$ in place of $F$ and $f$, respectively. 
The existence of $u^{(k)}$ follows from Theorem \ref{thm-simpl-eqn-2}. 
By the estimate \eqref{apr-est-0606} that we just proved, we have $u^{(k)}\in \sH^{1}_p(\Omega_T)$. 
Then, by the strong convergence of $\{F^{(k)}\}$ in $L_{p}(\Omega_T)^d$ and $\{f^{(k)}\}$ in $L_{p}(\Omega_T, x_d^{-p/2})$ and the linearity of the PDE in \eqref{x-d.model-eqn}, it is not too hard to show that
$\{u^{(k)}\}$ is a Cauchy sequence in $\sH^{1}_{p}(\Omega_T)$.
 Let $u\in \sH^{1}_{p}(\Omega_T)$ be its limit and by passing to the limit in the weak formulation defined in Definition \ref{weak-form}, it is easily seen that $u$ is a weak solution to the equation \eqref{x-d.model-eqn}.

\medskip

\noindent
{\bf Case II.} Consider $p \in (1, 2)$. 
We first prove the estimate \eqref{apr-est-0606}  assuming that $u \in \sH^1_p(\Omega_T)$ is a weak solution of \eqref{x-d.model-eqn}. The main idea is to use a duality argument. 
Though this is standard, we provide the details here as the equation \eqref{x-d.model-eqn} has some singularity and careful adjustments are needed. 
 
Let $q = {p}/(p-1) \in (2,\infty)$, $B \in L_q(\Omega_T)^d$ and $b \in L_q(\Omega_T, x_d^{-q/2})$. We consider the adjoint problem
\begin{equation} \label{adj-eqn-06}
x_d^{-1}(- \bar a_0 v_t + \lambda \bar c_0 v) - D_i\big( \overline{a}_{ji}(x_d) D_{j} v - B_{i}\chi_{(-\infty, T)}\big)  =  \lambda^{1/2} x_d^{-1} b\chi_{(-\infty, T)},
\end{equation}
in $\bR^{d+1}_+$ with the boundary condition $v = 0$ on $\partial \bR^{d+1}_+$.
By Case I, there exists a unique weak solution $v \in \sH^1_q(\bR \times \bR_+^d)$ of \eqref{adj-eqn-06} and
\begin{equation}\label{eq10.35}
\begin{split}
& \int_{\bR^{d+1}_+} \big(|Dv(z)|^q + \lambda^{1/2} |x_d^{-1/2}v(z)|^q \big) \,dz \\
& \leq N\int_{\Omega_T} \big(|B(z)|^q + |x_d^{-1/2}b(z)|^q \big) \, dz.
\end{split}
\end{equation}
Moreover, by the uniqueness of solutions, $v =0$ for $t \geq T$. Now,
by using $v$ as a test function for the equation \eqref{x-d.model-eqn} and $u$ as a test function for \eqref{adj-eqn-06}, we obtain
\[
\begin{split}
& \int_{\Omega_T}\big(B(z)\cdot D u(z)  + \lambda^{1/2} x_d^{-1} b(z) u(z) \big)\, dz \\
&= \int_{\Omega_T}\big(F(z)\cdot D v (z)+ \lambda^{1/2} x_d^{-1}f(z) v(z) \big)\, dz.
\end{split}
\]
Then, from H\"older's inequality and \eqref{eq10.35}, it follows that
\[
\begin{split}
& \left|\int_{\Omega_T}\big(B(z)\cdot D u(z) + \lambda^{1/2} x_d^{-1}b(z) u (z)\big)\, dz\right| \\
& \leq  \|F\|_{L_p(\Omega)} \|D v\|_{L_q(\Omega_T)} + \lambda^{1/2} \|f\|_{L_{p}(\Omega_T, x_d^{-p/2})} \| v\|_{L_q(\Omega_T, x_d^{-q/2})}\\
& \leq N\Big(\|F\|_{L_p(\Omega)} + \|f\|_{L_{p}(\Omega_T, x_d^{-p/2})}\Big)
\Big( \|B\|_{L_q(\Omega_T)} + \|b\|_{L_q(\Omega_T, x_d^{-q/2})} \Big).
\end{split}
\]
From this last estimate and as $B$ and $b$ are arbitrary, we obtain the a priori estimate \eqref{apr-est-0606}.

Finally, we prove the existence of solution $u \in \sH_p^1(\Omega_T)$. 
We follow the argument in \cite[Section 8]{MR3812104}. 
For $i=1,2,\ldots, d$ and $k \in \bN$, let
$$
F^{(k)}_i=\max\{-k,\min\{k,F_i\}\}\chi_{\widehat Q_k}(z),
$$
where $\hat{Q}_k$ is defined in \eqref{hat-Q}. Then $F^{(k)}\in L_2(\Omega_T)^{d}\cap L_p(\Omega_T)^{d}$. 
Also, by the dominated convergence theorem, $F^{(k)}\to F$ in $L_p(\Omega_T)^d$ as $k\to \infty$. 
Similarly, we define
\[
f^{(k)}(z) = x_d^{1/2}\max\{-k,\min\{k, x_d^{-1/2} f(z)\} \}\chi_{Q_k}(z),
\]
and we see that $\{f^{(k)}\}\subset L_2(\Omega_T, x_d^{-1})\cap L_p(\Omega_T, x_d^{-p/2})$ and $f^{(k)}\to f$ in $L_p(\Omega_T, x_d^{-p/2})$ as $k\to \infty$. 
By Theorem \ref{thm-simpl-eqn-2}, there is a unique weak solution $u^{(k)}\in \sH^{1}_{2}(\Omega_T)$ to the equation \eqref{x-d.model-eqn} with $F^{(k)}$ and $f^{(k)}$ in place of $F$ and $f$, respectively.

\medskip

\noindent
{\bf Claim A.} $u^{(k)} \in \sH_p^1(\Omega_T)$.

\smallskip
As the proof of this claim contains tedious calculations, and we provide it in Appendix \ref{claim-A-proof}. 
From the claim, and the a priori estimate that we just proved, we have
\[
\begin{split}
& \|Du^{(k)}\|_{L_p(\Omega_T)} + \sqrt{\lambda}\|u^{(k)}\|_{L_p(\Omega_T, x_d^{-p/2})} \\
& \leq N \big[\|F^{(k)}\|_{L_p(\Omega_T)} + \|f^{(k)}\|_{L_p(\Omega_T, x_d^{-p/2})} \big].
\end{split}
\] 
From this, the convergences of $F^{(k)}, f^{(k)}$, and the linearity of the PDE \eqref{x-d.model-eqn}, we see that the sequence $\{u^{(k)}\}_k$ is Cauchy in $\sH_p^1(\Omega_T)$. 
Let $u \in \sH_p^1(\Omega_T)$ be its limit. 
Then, it is easy to verify that $u$ is a weak solution of \eqref{x-d.model-eqn}. 
The theorem is proved.
\end{proof}
\section{Equations with partially VMO coefficients} \label{sec:proof-main}
The main goal of this section is to prove Theorem \ref{main-thrm} and Corollary \ref{cor-2}.
We first prove the following decomposition result which is similar to Corollary \ref{Simple-approx}.
\begin{lemma}\label{G-approx-propos}
Let $\gamma \in (0, 1)$,  $r \in (0, \infty)$,  $z_0\in \overline{\Omega}_T$, and  $q \in (2,\infty)$. 
Suppose that $A(z)=  |F(z)| + |x_d^{-1/2}f(z)| \in L_{2}(Q_{10r}^+(z_0))$ and $u \in \sH^{1}_q(Q_{10r}^+(z_0))$ is a weak solution of \eqref{eq11.52} in $Q_{10r}^+(z_0)$. 
Assume that  \eqref{a-BMO-cond} holds with some $\rho_0 >0$ and $\textup{spt}(u) \subset   (s - \rho_0r_0, s + \rho_0r_0) \times \bR^{d}_+$ for some $r_0>0$ and $s \in \bR$.
Then, we have
\[
u(t, x) = g(t, x) + h(t, x) \quad  \text{in}\  Q_{10r}^+(z_0),
\]
where $g$ and $h$ are functions in $\sH_2^1(Q_{10r}^+(z_0))$ that satisfy
\begin{align} \nonumber
\fint_{Q_{2r}^+(z_0)} |G|^2 \, dz &  \leq N \fint_{Q_{10r}^+(z_0)} |A(z)|^{2} \, dz   \\ \label{B-u-tilde-est-inter}
& \qquad +  N(\gamma^{1-2/q}+ r_0^{2-4/q}) \left(\fint_{Q_{10r}^+(z_0)} |Du|^q \, dz \right)^{2/q}
\end{align}
and
\begin{align} \label{D-L-infty-w-inter}
\|H\|_{L_\infty(Q_{r}^+(z_0))}^{2} \leq N \left( \fint_{Q_{10r}^+(z_0)} |U|^{q}\, dz\right)^{2/q} +  N\fint_{Q_{10r}^+(z_0)}  |A(z)|^{2} \, dz,
\end{align}
where $N = N(d, \nu,  q) >0$, and
\[
\begin{split}
& G(z)=|Dg(z)|+ \sqrt{\lambda}|x_d^{-1/2}g(z)|,\quad H(z)=|Dh(z)|+\sqrt{\lambda}|x_d^{-1/2}h(z)|, \\
& U=|D u(z)|+\sqrt{\lambda}|x_d^{-1/2}u(z)|.
\end{split}
\]
\end{lemma}

\begin{proof} Let $b = (b_1, b_2, \ldots, b_d)$ be defined by
\[
b_{i}(t,x) = \chi_{Q_{10r}^+(z_0)}(z) \big( a_{ij}(t,x) - {[a_{ij}]_{10r,z_0}}(x_d)\big)D_j u (t,x)  + F_i(z)\chi_{Q_{10r}^+(z_0)}(z),
\]
for $i=1,2,\ldots, d$, and let 
\[
\tilde{f}(z) =  \big[ f(z) + \sqrt{\lambda} ( [c_0]_{10r, z_0}(x_d) -c_0(z)) u \big] \chi_{Q_{10r}^+(z_0)}(z)
\]
where $[a_{ij}]_{10r, z_0}$ and $[c]_{10r, z_0}$ are defined in Definition \ref{a-ossi-def}.  
We claim that $b \in L_2(\Omega_T)$ and $\tilde{f} \in L_2(\Omega_T, x_d^{-1})$.  
Indeed, if $r \in (0,\rho_0/10)$,  by H\"{o}lder's inequality, \eqref{a-BMO-cond}, and \eqref{elli}, we see 
that
\begin{align*}
 & \fint_{Q_{10r}^+(z_0)} |b(z)|^2 \, dz  \\
& \leq \left(\fint_{Q_{10r}^+(z_0)} |a_{ij} -[a_{ij}]_{10r,z_0}|^{\frac{2q}{q-2}} \, dz \right)^{\frac{q-2}{q}} \left(\fint_{Q_{10r}^+(z_0)} |D u|^{q} \, dz \right)^{\frac{2}{q}} \\ \nonumber
& \qquad +  \fint_{Q_{10r}^+(z_0)} | F |^{2} \, dz \nonumber\\
& \leq N \gamma^{\frac{q-2}{q}} \left(\fint_{Q_{10r}^+(z_0)} | D u |^q \, dz \right)^{2/q} + \fint_{Q_{10r}^+(z_0)} | F |^{2} \, dz.
\end{align*}
Similarly, using \eqref{a-c.conds} and H\"{o}lder's inequality, we obtain
\begin{align*}
& \fint_{Q_{10r}^+(z_0)} |x_d^{-1/2} \tilde{f}(z)|^2 \, dz \\
& \leq N  \gamma^{\frac{q-2}{q}}  \left(\fint_{Q_{10r}^+(z_0)} |\sqrt{\lambda} x_d^{-1/2}u(z)|^q \, dz \right)^{2/q} + \fint_{Q_{10r}^+(z_0)} |x_d^{-1/2} f(z)|^{2} \, dz.
\end{align*}
On the other hand, when $r \geq \rho_0/10$, as $\text{spt}(u) \subset  (s - \rho_0r_0, s + \rho_0r_0) \times \bR^{d}_+$ and by the boundedness of $(a_{ij})$ in \eqref{elli}, we have
\begin{align*} 
& \fint_{Q_{10r}^+(z_0)} |b(z)|^2 \, dz  \\
& \leq N \left(\fint_{Q_{10r}^+(z_0)}  \chi_{(s - \rho_0r_0, s + \rho_0r_0}(t)  \, dz \right)^{\frac{q-2}{q}} \left(\fint_{Q_{10r}^+(z_0)} |D u|^{q} \, dz \right)^{\frac{2}{q}} \\
& \qquad + N \fint_{Q_{10r}^+(z_0)} | F |^{2} \, dz  \nonumber\\
& \leq N \Big(\frac{\rho_0 r_0}{r} \Big)^{\frac{q-2}{q}} \left(\fint_{Q_{10r}^+(z_0)} | D u |^q \,dz \right)^{2/q} +  N \fint_{Q_{10r}^+(z_0)} | F |^{2} \, dz \\
& \leq N r_0^{\frac{q-2}{q}} \left(\fint_{Q_{10r}^+(z_0)} | D u |^q \, dz \right)^{2/q} +  N \fint_{Q_{10r}^+(z_0)} | F |^{2} \, dz.
\end{align*}
By the same way using \eqref{a-c.conds}, we also infer that
\[
\begin{split}
& \fint_{Q_{10r}^+(z_0)} |x_d^{-1/2} \tilde{f}(z)|^2 \, dz \\
& \leq N  r_0^{\frac{q-2}{q}}  \left(\fint_{Q_{10r}^+(z_0)} |\sqrt{\lambda} x_d^{-1/2}u(z)|^q \, dz \right)^{2/q} + \fint_{Q_{10r}^+(z_0)} |x_d^{-1/2} f |^{2} \, dz.
\end{split}
\]
Summing up, we see that
\begin{align} \notag 
& \fint_{Q_{10r}^+(z_0)}\big(|b(z)|^2 + |x_d^{-1/2} \tilde{f}(z)|^2\, \big)dz \\ \label{G-inter-est}
 & \leq N \Big(r_0^{\frac{q-2}{q}} + \gamma^{\frac{q-2}{q}}\Big)
 \left(\fint_{Q_{10r}^+(z_0)} |U(z) |^q \, dz \right)^{2/q} + N \fint_{Q_{10r}^+(z_0)} |A(z)|^{2} \, dz,
\end{align}
for every $r \in (0, \infty)$.

\smallskip
Now, observe that $u \in \sH_q^1(Q_{10r}^+(z_0))$ is weak solution of
\begin{equation*} 
x_d^{-1}  (\partial_t u +\lambda [c_0]_{10r, z_0}(x_d) u) -  D_i \big([a_{ij}]_{10r,z_0}(x_d) D_{j} u - b_i \big)   = \lambda^{1/2} x_d^{-1} \tilde{f},
\end{equation*}
in $Q_{10r}^+(z_0)$ with the boundary condition $u =0$ on $Q_{10r}(z_0) \cap \{x_d =0\}$ when $Q_{10r}(z_0) \cap \{x_d =0\} \not=\emptyset$. From Definition \ref{a-ossi-def}, \eqref{elli}, and \eqref{a-c.conds},  it follows that the coefficients $[a_{ij}]_{10r,z_0}$ and $[c_0]_{10r, z_0}(x_d)$ satisfy the assumptions \eqref{a-c.cond}, \eqref{elli-cond}, and \eqref{structure.cond}. 
Then, by applying Corollary \ref{Simple-approx} and using \eqref{G-inter-est}, the assertions in Lemma \ref{G-approx-propos} follow.
\end{proof}

Now, we are ready to prove Theorem \ref{main-thrm}.

\begin{proof}[Proof of Theorem \ref{main-thrm}] 
Similar to the proof of Theorem \ref{thm-simpl-eqn}, we only need to consider  $p \in (2, \infty)$.  
We first prove the a priori estimate \eqref{main-est-0508} assuming that $u \in \sH_p^1(\Omega_T)$ is a weak solution of \eqref{main-eqn}-\eqref{main-bdr-cond}. We suppose that $\lambda>0$ and assume for a moment that
$$
\textup{spt}(u) \subset  (s - \rho_0r_0, s + \rho_0r_0) \times \bR^{d}_+
$$
with some $s \in (-\infty, T)$ and some $r_0 \in (0,1)$. We claim that \eqref{main-est-0508}  holds when $\gamma$ and  $r_0$ are sufficiently small depending on $d$, $\nu$, and $p$.
 Let $q \in (2, p)$ be fixed. 
 By Lemma \ref{G-approx-propos}, for each $r >0$ and $z_0 \in \overline{\Omega}_T$, we can write
\[
u(t, x) = g(t, x) + h(t, x) \quad  \text{in}\ Q_{2r}^+(z_0),
\]
where $g$ and $h$ satisfy \eqref{B-u-tilde-est-inter} and \eqref{D-L-infty-w-inter}.  
Then it follows from the standard real variable argument,  see \cite{DK11b} for example, that
\[
\begin{split}
& \|Du\|_{L_p(\Omega_T)} + \sqrt{\lambda} \|u\|_{L_p(\Omega_T, x_d^{-p/2})} \\
&  \leq N(\gamma^{1-2/q} + r_0^{1-2/q})\big( \|Du\|_{L_{p}(\Omega_T)} + \sqrt{\lambda} \|u\|_{L_p(\Omega_T, x_d^{-p/2})}\big) \\
& \qquad + N\|F\|_{L_p(\Omega_T)} +N\|f\|_{L_p(\Omega_T, x_d^{-p/2})}
\end{split}
\]
for $N = N(d, \nu, p)>0$. 
Then, by choosing $\gamma$ and $r_0$ sufficiently small so that $N(\gamma^{1-2/q} + r_0^{1-2/q}) < 1/2$, we obtain \eqref{main-est-0508}.

Next, we remove the assumption that $\textup{spt}(u) \subset  (s -\rho_0r_0, s+ \rho_0r_0) \times \bR^{d}_+$. 
The idea is to use a partition of unity argument along the time variable. 
Though it is standard, adjustments are needed and we provide details for completeness. 
Let
$$
\xi=\xi(t) \in C_0^\infty((-\rho_0r_0, \rho_0r_0))
$$
be a standard non-negative cut-off function satisfying
\begin{equation} \label{xi-0515}
\int_{\bR} \xi(s)^p\, ds =1, \quad  \int_{\bR}|\xi'(s)|^p\,ds \leq \frac{N}{(\rho_0r_0)^{p}}.
\end{equation}
For any $s \in (-\infty,  \infty)$, let $u^{(s)}(z) = u(z) \xi(t-s)$ for $z = (t, x) \in \Omega_T$. 
Then $u^{(s)} \in \sH_p^1(\Omega_T)$ is a weak solution of
\[
x_d^{-1}( u^{(s)}_t + \lambda c_0(z) u^{(s)}) -D_i\big(a_{ij} D_j u^{(s)} - F^{(s)}_{i}\big)  = \lambda^{1/2} x_d^{-1} f^{(s)}
\]
in $\Omega_T$ with the boundary condition $u^{s} =0$ on $\{x_d =0\}$, where
\[
F^{(s)}(z) = \xi(t-s) F(z), \quad f^{(s)}(z)   = \xi(t-s) f(z)  +  \lambda^{-1/2}\xi'(t-s) u(z).
\]
As $\text{spt}(u^{(s)}) \subset (s -\rho_0r_0, s+ \rho_0r_0) \times \bR^{d}_{+}$, we can apply the estimate we just proved to infer that
\[
\begin{split}
& \|Du^{(s)}\|_{L_p(\Omega_T)} + \sqrt{\lambda} \|u^{(s)}\|_{L_p(\Omega_T, x_d^{-p/2})}  \\
& \leq N \|F^{(s)}\|_{L_p(\Omega_T)} +N\|f^{(s)}\|_{L_p(\Omega_T, x_d^{-p/2})}.
\end{split}
\]
Then, by integrating the $p$-power of this estimate with respect to $s$, we get
\begin{align}\notag 
& \int_{\bR}\Big( \|Du^{(s)}\|_{L_p(\Omega_T)}^p + \lambda^{p/2} \|u^{(s)}\|^p_{L_p(\Omega_T, x_d^{-p/2})}\Big)\, ds\\  \label{int-0515}
&  \leq N\int_{\bR} \Big( \|F^{(s)}\|^p_{L_p(\Omega_T)} + \|f^{(s)}\|^p_{L_p(\Omega_T, x_d^{-p/2})} \Big)\, ds.
\end{align}
By Fubini's theorem and \eqref{xi-0515}, it follows that
\[
\int_{\bR}\|Du^{(s)}\|_{L_p(\Omega_T)}^p\, ds = \int_{\Omega_T}\int_{\bR} |Du(z)|^p \xi^p(t-s)\, dsdz  = \|Du\|_{L_p(\Omega_T)}^p,
\]
and
\[
\begin{split}
& \int_{\bR}\|u^{(s)}\|_{L_p(\Omega_T,x_d^{-p/2})}^p\, ds = \|u\|_{L_p(\Omega_T, x_d^{-p/2})}^p,  \\
& \int_{\bR}\|F^{(s)}\|_{L_p(\Omega_T)}^p\, ds = \|F\|_{L_p(\Omega_T)}^p.
\end{split}
\]
On the other hand, as $r_0$ depends only on $d$, $\nu$, and $p$, from the definition of $f^{(s)}$, \eqref{xi-0515}, and the Fubini theorem, we have
\[
\begin{split}
\left(\int_{\bR} \|f^{(s)}\|_{L_p(\Omega, x_d^{-p/2})}^p\, ds \right)^{1/p} \leq N  \|f\|_{L_p(\Omega_T, x_d^{-p/2})} +  N\rho_0^{-1} \lambda^{1/2}\|u\|_{L_p(\Omega_T, x_d^{-p/2})}
\end{split}
\]
for $N = N(d, \nu, p)$. 
Then, by combining the estimates we just derived, we infer from \eqref{int-0515} that
\[
\begin{split}
& \|Du\|_{L_p(\Omega_T)} + \sqrt{\lambda} \|u\|_{L_p(\Omega_T, x_d^{-p/2})}\\
&  \leq N\|F \|_{L_p(\Omega_T)} +N\|f\| _{L_p(\Omega_T, x_d^{-p/2})} + N\rho_0^{-1}\lambda^{-1/2}\|u\|_{L_p(\Omega_T, x_d^{-p/2})}
\end{split}
\]
with $N=N(d, \nu, p)$. 
Now we choose $\lambda_0 = 2N$. 
Then, with $\lambda \geq \lambda_0 \rho_0^{-1}$, we have $N\rho_0^{-1}\lambda^{-1/2} \leq \sqrt\lambda/2$, and consequently
\[
\begin{split}
& \|Du \|_{L_p(\Omega_T)}
+ \sqrt{\lambda} \|u \|_{L_p(\Omega_T, x_d^{-p/2})}\\
&  \leq   N \|F \|_{L_p(\Omega_T)} +N \|f\| _{L_p(\Omega_T, x_d^{-p/2})}+\frac{\sqrt{\lambda}}{2} \|u\|_{L_p(\Omega_T, x_d^{-p/2})}.
\end{split}
\]
This estimate yields \eqref{main-est-0508}.

We finally prove the  existence of the solution $u \in \sH_p^1(\Omega_T)$. 
We note that for the class of equations with constant coefficients
\[
x_d^{-1}(u_t + \lambda u) - D_i(D_i u - F_i)  = \lambda^{1/2} x_d^{-1} f 
\]
in $\Omega_T$ with the boundary condition $u=0$ on $\{x_d =0\}$, the existence and uniqueness of solution in $\sH_p^1(\Omega_T)$ is proved in Theorem \ref{thm-simpl-eqn}. 
From this fact, and the method of continuity, we can easily derive the existence of solution for \eqref{main-eqn}-\eqref{main-bdr-cond}. 
The proof is now completed.
\end{proof}

Next, we give the proof of Corollary \ref{cor-2}.

\begin{proof}[Proof of Corollary \ref{cor-2}] 
Recall that we consider the equation
\begin{equation} \label{simplest-eqn-1}
\left\{
\begin{array}{cccl}
u_t + u - x_d \Delta u & = & f &\quad \text{in}  \quad \Omega_T, \\
u & =& 0 & \quad \text{on} \quad \{x_d =0\}.
\end{array} \right.
\end{equation}
We note that when coefficients are constant, with a scaling argument, Theorem \ref{main-thrm} holds for all $\lambda>0$. 
See Theorem \ref{thm-simpl-eqn}  for a similar result with slightly more general equations. 
Then, applying \eqref{main-est-0508} for \eqref{simplest-eqn-1}, we have
\begin{equation} \label{Du-Lp--613}
\|Du\|_{L_p(\Omega_T)} +  \|u\|_{L_p(\Omega_T, x_d^{-p/2})} \leq  N \|f\|_{L_p(\Omega_T, x_d^{-p/2})}.
\end{equation}
Similarly, as $u_t$ solves the same equation as $u$ in \eqref{simplest-eqn-1} with $f$ replaced by $f_t$, we have
\begin{equation} \label{ut-Lp-0613-l}
\|Du_t\|_{L_p(\Omega_T)} +   \|u_t \|_{L_p(\Omega_T, x_d^{-p/2})} \leq  N \|f_t\|_{L_p(\Omega_T, x_d^{-p/2})}.
\end{equation}
Now, let $g(z) = x_d^{-1}[f- u_t - u ]$. For each fixed $t \in (-\infty, T)$, it follows from the  PDE in \eqref{simplest-eqn-1} that $u = u(t, \cdot)$
is a solution of the Poisson equation
\begin{equation} \label{elli-eqn}
\left\{
\begin{array}{cccl}
-\Delta u &=& g & \quad \text{ in } \quad \bR^d_+ ,\\
 u &=& 0 &\quad \text{ on } \partial \bR^d_+.
\end{array} \right.
\end{equation}
As  $x_d^{p/2}$ is an $A_p$ Muckenhoupt weight when $p>2$. We apply the classical weighted Calder\'{o}n-Zydmund estimate for this elliptic equation \eqref{elli-eqn}, and then integrating the result in the time variable to obtain
\begin{equation} \label{DDu-weight}
\begin{split}
\|D^2 u\|_{L_p(\Omega_T, x_d^{p/2})} & \leq N  \|g\|_{L_p(\Omega_T, x_d^{p/2})} \\
 & \leq N \big[ \|f_t\|_{L_p(\Omega_T, x_d^{-p/2})} + \|f\|_{L_p(\Omega_T, x_d^{-p/2})} \big]
\end{split}
\end{equation}
for $N = N(d, p)>0$ and $p >2$, where we also used \eqref{Du-Lp--613} and \eqref{ut-Lp-0613-l} in the last estimate. 
Observe also that \eqref{DDu-weight} also holds when $p=2$ by an energy estimate using integration by parts. 
In fact, we write
\[
-x_d \Delta u = h \quad \text{in} \quad \Omega_T
\]
where $h = f - u_t -u$. Then, multiplying this equation with $D_{x'}^2 u$ and using the integration by parts, we obtain
\[
\int_{\Omega_T} |DD_{x'} u|^2 x_d \,dz \leq \int_{\Omega_T} |h(z)| |D^2_{x'} u| \,dz.
\]
From this, we obtain
\[
\int_{\Omega_T} |DD_{x'} u|^2 x_d \,dz \leq N \int_{\Omega_T} \big[|u_t|^2 + |u|^2 + |f|^2 \big]x_d^{-1} \,dz.
\]
On the other hand, from the PDE in \eqref{simplest-eqn-1}, we also obtain
\[
\int_{\Omega_T} |D^2_{d} u|^2 x_d \,dz \leq N \int_{\Omega_T} \big[|u_t|^2 + |u|^2 + |f|^2 \big]x_d^{-1} \,dz + \int_{\Omega_T} |D_{x'}^2 u|^2 x_d \,dz.
\]
Then, combining these estimates, we obtain \eqref{DDu-weight} for $p=2$. 
Finally, by collecting all estimates \eqref{Du-Lp--613}, \eqref{ut-Lp-0613-l} and \eqref{DDu-weight}, we get \eqref{est-model-eqn}. 
The proof is completed.
\end{proof}

\appendix
\section{Proof of Claim A}  \label{claim-A-proof} 
We follow an approach used in \cite{Dong-Phan}. 
For a fixed $k \in \bN$, let us recall the definition of $\hat{Q}_k$ given in \eqref{hat-Q}. Since $u^{(k)}\in \sH^{1}_{2}(\Omega_T)$, $\hat{Q}_{2k}$ is compact, and $p \in (1,2)$, by H\"older's inequality,
\begin{equation*} 
 \|u^{(k)}\|_{L_p(\widehat Q_{2k}, x_d^{-p/2})}+\|Du^{(k)}\|_{L_p(\widehat Q_{2 k})}<\infty.
\end{equation*}
Hence, we only need to prove that
\[ \|u^{(k)}|\|_{L_p(\Omega_T\setminus \widehat{Q}_{2k}, x_d^{-p/2})}  +\|Du^{(k)}\|_{L_p(\Omega_T\setminus \widehat{Q}_{2k})} <\infty. 
\]
The main idea is to use Theorem \ref{thm-simpl-eqn-2} and a localization technique. For $j\ge 0$,  let $\eta_j$ be such that
\begin{align*}
\eta_j&\equiv 0\quad \text{in}\,\, \widehat Q_{2^jk},  \quad \eta_j \equiv 1 \quad \text{outside}\,\, \widehat Q_{2^{j+1}k},
\end{align*}
and $|D \eta_j|+|(\eta_j)_t|\le N_0 2^{-j}$, where $N_0$ is independent of $j$.  

We note that the supports of $F^{(k)}$ and $f^{(k)}$ are in $\widehat{Q}_{{k}}$, while the supports of $\eta_j$ are outside $\widehat{Q}_{{k}}$.  
Therefore, $\eta_j F_i^{(k)} \equiv \eta_j f^{(k)}  \equiv F_i^{(k)} D_i\eta_j \equiv 0$ for every $i = 1, 2,\ldots, d$ and $j \geq 0$. 
Because of this and with a simple calculation, we see that $w^{(k,l)}:=u^{(k)}\eta_l\in \sH^{1}_{2}(\Omega_T)$ is a weak solution of
\begin{equation*}
x_d^{-1} \big( \overline{a}_0 w^{(k,l)}_t  +\lambda  \overline{c}_0 w^{(k,l)}\big) - D_i\big (\overline{a}_{ij} D_j w^{(k,l)} - F^{(k,l)}_i\big)  =  \lambda^{1/2} x_d^{-1} f^{(k,l)} 
\end{equation*}
in $\Omega_T$, with the boundary condition $w^{(k,l)} =0$ on $\{x_d =0\}$, where
\[
F^{(k,l)}_i  =  u^{(k)}\overline{a}_{ij}D_j \eta_l, \quad i = 1, 2,\ldots, d,
\]
and
\[
 f^{(k,l)} = \lambda^{-1/2}\big(\overline{a}_0 u^{(k)}(\eta_l)_t - x_d\overline{a}_{ij} D_j u^{(k)} D_i\eta_l \big).
 \]
It then follows from Theorem \ref{thm-simpl-eqn-2} that
\begin{align*}
& \|Dw^{(k,l)} \|_{L_2(\Omega_T)}+ \sqrt{\lambda} \|w^{(k,l)}\|_{L_2(\Omega_T, x_d^{-1})}\\
& \le N \|F^{(k,l)}\|_{L_2(\Omega_T)}+N
\|f^{(k,l)}\|_{L_2(\Omega_T, x_d^{-1})} , \quad \forall \ k, l \in \bN.
\end{align*}
This and as $x_d \sim 2^{j}$ in $Q_{2^{j+2}k}\setminus \widehat Q_{2^{j+1}k}$, we see that
\begin{align*}
&\|Du^{(k)}\|_{L_2(\widehat Q_{2^{j+2}k}\setminus \widehat Q_{2^{j+1}k})}+\sqrt{\lambda} \|u^{(k)}\|_{L_2(\widehat Q_{2^{j+2}k}\setminus \widehat Q_{2^{j+1}k}, x_d^{-1})}\\
&\le
N\Big( 2^{-j/2}\|u^{(k)}\|_{L_2(\widehat Q_{2^{j+1}k}\setminus \widehat Q_{2^{j}k}, x_d^{-1})}+ \lambda^{-1/2}2^{-j}
\|u^{(k)}\|_{L_2(\widehat Q_{2^{j+1}k}\setminus \widehat Q_{2^{j}k}, x_d^{-1})}\\
&\quad + \lambda^{-1/2}2^{-j/2}\|Du^{(k)}\|_{L_2(\widehat Q_{2^{j+1}k}\setminus \widehat Q_{2^{j}k})} \Big)\\
&\le C2^{-j/2}\big(\|Du^{(k)}\|_{L_2(\widehat Q_{2^{j+1}k}\setminus \widehat Q_{2^{j}k})}+\sqrt{\lambda} \|u^{(k)}\|_{L_2(\widehat Q_{2^{j+1}k}\setminus \widehat Q_{2^{j}k}, x_d^{-1})}\big)
\end{align*}
for every $j \geq 1$, where $C = C(\lambda, k)>0$. Now, we iterate this last estimate to get 
\begin{align}  \label{eq9.01}
  &\|Du^{(k)}\|_{L_2(\widehat Q_{2^{j+1}k}\setminus \widehat Q_{2^{j}k})}+\sqrt{\lambda} \|u^{(k)}\|_{L_2(\widehat Q_{2^{j+1}k}\setminus \widehat Q_{2^{j}k}, x_d^{-1})}\nonumber\\
&\le C^j2^{-j(j-1)/4}\big(\|Du^{(k)}\|_{L_2(\widehat Q_{2k})}+ \sqrt{\lambda} \|u^{(k)}\|_{L_2(\widehat Q_{2k}, x_d^{-1})}\big).
\end{align}
Then, it follows from \eqref{eq9.01} and H\"{o}lder's inequality that
\begin{align*}
&\|Du^{(k)}\|_{L_p(\widehat Q_{2^{j+1}k}\setminus \widehat Q_{2^{j}k})}+\sqrt{\lambda} \|u^{(k)}\|_{L_p(\widehat Q_{2^{j+1}k}\setminus \widehat Q_{2^{j}k},x_d^{-p/2})}\\
&\le |\widehat Q_{2^{j+1}k}|^{\frac 1 p-\frac 1 2}\big(\|Du^{(k)}\|_{L_2(\widehat Q_{2^{j+1}k}\setminus \widehat Q_{2^{j}k})}+ \sqrt{\lambda} \|u^{(k)}\|_{L_2(\widehat Q_{2^{j+1}k}\setminus \widehat Q_{2^{j}k}, x_d^{-1})}\big)\\
&\le N^j2^{-\frac {j(j-1)} 4}\big(\|Du^{(k)}\|_{L_2(\widehat Q_{2k})}+\sqrt{\lambda} \|u^{(k)}\|_{L_2(\widehat Q_{2k}, x_d^{-1})}\big),
\end{align*}
where $N = N(d, k, p, \lambda)>0$. Therefore,
\[
\begin{split}
&\|Du^{(k)}\|_{L_p(\Omega_T \setminus \widehat Q_{2k})}+\sqrt{\lambda} \|u^{(k)}\|_{L_p(\Omega_T \setminus \widehat Q_{2k}, x_d^{-p/2})}
\\
&=\sum_{j=1}^\infty \Big( \|Du^{(k)}\|_{L_p(\widehat Q_{2^{j+1}k}\setminus \widehat Q_{2^{j}k})}+\sqrt{\lambda} \|u^{(k)}\|_{L_p(\widehat Q_{2^{j+1}k}\setminus \widehat Q_{2^{j}k}, x_d^{-p/2})}\Big)\\
& \leq N \|Du^{(k)}\|_{L_2(\widehat Q_{2k})}
+N\sqrt{\lambda} \|u^{(k)}\|_{L_2(\widehat Q_{2k}, x_d^{-1})} < \infty,
\end{split}
\]
and this proves the claim that $u^{(k)} \in \sH_p^1(\Omega_T)$.

\section{Proof of Theorem \ref{zero-trace-1}} \label{sec-proof-trace}
For $x_d \in (0,1)$ and $y \in (0, x_d)$, by the fundamental theorem of calculus and H\"{o}lder's inequality, we have
\[
\begin{split}
|u(x', x_d)|  & \leq |u(x', y)| + \int_{y}^{x_d} |D_du(x', s)| \,ds\\
& \leq |u(x', y)| + x_d^{\frac{p-1}{p}} \left(\int_{0}^{x_d} |D_du(x', s)|^p\, ds\right)^{1/p}.
\end{split}
\]
Then, by integrating this inequality in $y$ on $(0, x_d)$, we obtain
\[
\begin{split}
& |u(x', x_d)|x_d  \leq \int_0^{x_d}|u(x', y)| \, dy +  x_d^{2-\frac{1}{p}} \left(\int_{0}^{x_d} |D_du(x', s)|^p  \, ds\right)^{1/p} \\
& \leq N x_d^{\frac{3}{2}-\frac{1}{p}}\left(\int_0^{x_d}|u(x', y)|^p y^{-p/2}\, dy \right)^{1/p}+  x_d^{2-\frac{1}{p}} \left(\int_{0}^{x_d} |D_du(x', s)|^p  \, ds\right)^{1/p},
\end{split}
\]
where we used H\"{o}lder's inequality again in the last estimate. 
Therefore, we have
\[
|u(x', x_d)|  \leq N x_d^{\frac{1}{2} -\frac{1}{p}}\left(\int_0^{x_d}|u(x', y)|^p y^{-p/2}\, dy \right)^{1/p} 
+ x_d^{1-\frac{1}{p}} \left(\int_{0}^{x_d} |D_du(x', s)|^p\, ds\right)^{1/p}
\]
for $N = N(p) >0$. Now, we take the $L_p$-norm of the last estimate in $B_1'$, we infer that
\[
\|u(\cdot, x_d)\|_{L_p(B_1')} \leq N x_d^{\frac{1}{2} -\frac{1}{p}}\|u\|_{W^1_p(B_{1}' \times [0,x_d])}.
\]
Then, sending $x_d \rightarrow 0^+$ and using the fact that $p \geq 2$, we obtain
\[
u(x', 0) =0, \quad \text{for a.e.} \quad x' \in B_1'.
\]
Next, we prove $W^1_p(\cD) = \sW^{1}_p (\cD)$. 
By definition, $ \sW^{1}_p (\cD) \subset W^1_p(\cD)$.
Therefore, it suffices to prove that for $u \in W^1_p(\cD)$, there is a sequence of functions $\{u_k\} \subset W^1_p(\cD)$ that vanish near $\{x_d =0\}$ such that $u_k \rightarrow u$ in $W^1_p(\cD)$.
Let $\zeta \in C^\infty([0,\infty))$ be such that
\[
0\leq \zeta \leq 1, \quad \zeta=1  \text{ on } [0,1], \quad \zeta =0  \text{ in } (2,\infty).
\]
Denote by $\zeta_k(x) = \zeta(k x_d)$, and $u_k(x) = u(x)(1-\zeta_k(x))$ for $x\in \bR_+^d$, and $k\in \bN$.
Then,
\[
D_{x'} u_k(x) = (1-\zeta_k(x)) D_{x'} u(x), \quad D_d u_k(x) = (1-\zeta_k(x)) D_{d} u(x) - k \zeta'(kx_d) u(x).
\]
Consequently, 
\begin{align*}
&\|Du_k-Du\|_{L_{p}(\cD)} \\
\leq \ & N \|\zeta_k Du\|_{L_{p}(\cD)}  + N k\left( \int_0^{2/k} \int_{B_1' } |u(x',x_d)|^p \,dz' dx_d \right)^{1/p}\\
\leq \ & N \|\zeta_k Du\|_{L_{p}(\cD)}  \\
&+ N k \left(\int_0^{2/k} x_d^{p-1}\,dx_d\right)^{1/p} \left(\int_0^{2/k} \int_{B_1'} |Du(x',x_d)|^p\,dz\right)^{1/p}\\
\leq \ & N  \left( \int_0^{2/k} \int_{B_1' } |Du(x',x_d)|^p \,dz\right)^{1/p}.
\end{align*}
In the second last inequality above, we used  the fundamental theorem of calculus and H\"{o}lder's inequality to yield
\begin{align*}
|u(x', x_d)| &\leq  \int_{0}^{x_d} |D_d u(x', s)|\,ds
\leq x_d^{\frac{p-1}{p}} \left(\int_{0}^{x_d} |D_du(x', s)|^p \, ds\right)^{1/p}.
\end{align*}
Let $k\to \infty$ to complete the proof.

\end{document}